\documentclass[11pt]{amsart}%%for Takeda
\makeatletter

\@addtoreset{equation}{section}
\makeatother
\usepackage[margin=3cm]{geometry}
\newenvironment{nouppercase}{%
  \renewcommand{\uppercasenonmath}[1]{}}{}
\usepackage[dvipdfmx]{graphicx}
\usepackage{amsmath,amssymb,amsthm,color}
\usepackage{mathrsfs} 
\usepackage[T1]{fontenc}
\usepackage{ytableau}
\usepackage{tikz}
\usepackage{xcolor}
\allowdisplaybreaks[2]
\theoremstyle{definition}
\newtheorem{theorem}[equation]{Theorem}
\newtheorem{prop}[equation]{Proposition} 
 
\newtheorem*{conj*}{Conjecture}
\newtheorem*{theorem*}{Theorem}
\newtheorem{rmk}[equation]{Remark} 
\newtheorem{cor}[equation]{Corollary}

\newtheorem{lemma}[equation]{Lemma}
\newtheorem{exam}[equation]{Example}
\def\zetas{\zeta^{\star}}
\def\jt{{\rm JT}}
\def\jte{{\rm JTE}}
\def\jts{{\rm JT}^\star}
\def\jtse{{\rm JTE}^\star}
\begin{document}

\title[Pieri formulas for hook type]{The Pieri formulas for hook type Schur multiple zeta functions}
\author[M. Nakasuji]{Maki Nakasuji}
\address{Department of Information and Communication Science, Faculty of Science, Sophia
University, 7-1 Kioi-cho, Chiyoda-ku, Tokyo, 102-8554, Japan}
\email{ nakasuji@sophia.ac.jp}
\author[W. Takeda]{Wataru Takeda}
\address{Department of Applied Mathematics, Tokyo University of Science,
1-3 Kagurazaka, Shinjuku-ku, Tokyo 162-8601, Japan.}
\email{w.takeda@rs.tus.ac.jp}

\begin{abstract}
We study the Pieri type formulas for the Schur multiple zeta functions along with those for the Schur polynomials. To formulate these formulas, we introduce a new insertion rule for adding boxes in the Young tableaux and obtain the results for the hook type Schur multiple zeta functions. For the proof, we show {certain} extended Jacobi-Trudi formulas for the Schur multiple zeta functions. 
\end{abstract}

\begin{nouppercase}
\maketitle
\end{nouppercase}

\section{Introduction}
The {\it Schur multiple zeta functions} introduced by Nakasuji-Phuksuwan-Yamasaki \cite{npy} as a generalization of both multiple zeta and multiple zeta-star functions of Euler-Zagier type.
These are defined as sum like usual Schur polynomials of multiple zeta functions.
More precisely,
for any partition $\lambda$, i.e. a non-decreasing sequence $( \lambda_1, \ldots, \lambda_n)$ of { positive} integers,
we associate Young diagram
$D(\lambda)=\{(i, j)\in {\mathbb Z}^2 ~|~ 1\leq i\leq n, 1\leq j\leq \lambda_i\}$ depicted as a collection of square boxes with the $i$-th row has $\lambda_i$ boxes, as with matrices.
For a partition $\lambda$, a Young tableau $T=(t_{ij})$ of shape $\lambda$ over a set $X$ is a filling of $D(\lambda)$ obtained by putting $t_{ij}\in X$ into $(i,j)$ box of $D(\lambda)$.
 We denote by $T(\lambda,X)$ the set of all Young tableaux of shape $\lambda$ over $X$ and denote by $SSYT(\lambda)$ the set of semi-standard Young tableaux {$(t_{ij})\in T(\lambda,\mathbb N)$ which satisfies} weakly increasing from left to right in each row $i$,
and strictly increasing from top to bottom in each column $j$.
Then for a given set ${ \pmb s}=(s_{ij})\in T(\lambda,\mathbb{C})$ of variables, 
{\it Schur multiple zeta functions} of {shape} $\lambda$ are defined to be
\[
\zeta_{\lambda}({ \pmb s})=\sum_{M\in SSYT(\lambda)}\frac{1}{M^{ \pmb s}},
\]
where $M^{ \pmb s}=\displaystyle{\prod_{(i, j)\in D(\lambda)}m_{ij}^{s_{ij}}}$ for $M=(m_{ij})\in SSYT(\lambda)$. 
The function $\zeta_\lambda(\pmb s)$ absolutely converges in \[
 W_{\lambda}
=
\left\{{\pmb s}=(s_{ij})\in T(\lambda,\mathbb{C})\,\left|\,
\begin{array}{l}
 \text{$\Re(s_{ij})\ge 1$ for all $(i,j)\in D(\lambda) \setminus C(\lambda)$} \\[3pt]
 \text{$\Re(s_{ij})>1$ for all $(i,j)\in C(\lambda)$}
\end{array}
\right.
\right\}
\]
with $C(\lambda)$ being the set of all corners of $\lambda$. In this article, we assume that all variables of $\zeta_\lambda$ are elements of $W_\lambda$. We sometimes write $\pmb s$ as $\zeta_{\lambda}({ \pmb s})$ for short if there is no confusion.
We can regard these functions as generalizations of both multiple zeta functions $\zeta(s_1,\ldots,s_r)$ and multiple zeta star functions $\zetas(s_1,\ldots,s_r)$ defind by
\[
 \zeta(s_1,\ldots,s_r)=\sum_{1\le n_1<\cdots< n_r}\frac{1}{n_1^{s_1},\ldots,n_r^{s_r}},\
     \zetas(s_1,\ldots,s_r)=\sum_{1\le n_1\le\cdots\le n_r}\frac{1}{n_1^{s_1},\ldots,n_r^{s_r}},
 \]
since
\begin{equation}
    \label{zetazetas}
\ytableausetup{boxsize=normal,aligntableaux=center}
 \zeta_{(\{1\}^r)}\left(\begin{ytableau}
  s_1\\
  \vdots\\
  s_r
\end{ytableau}\right)=\zeta(s_1,\ldots,s_r),\ 
     \zeta_{(r)}\left(\begin{ytableau}
  s_1&\cdots&s_r
\end{ytableau}\right)=\zetas(s_1,\ldots,s_r).
\end{equation}

{Since the Schur multiple zeta functions have the structure of the Schur polynomials, they} are expected to have similar properties to those for Schur polynomials. Indeed, in \cite{npy}, Nakasuji-Phuksuwan-Yamasaki showed we can obtain some determinant formulas similar to those for the Schur polynomials under some assumptions, such as the Jacobi-Trudi \cite{ma}, Giambelli \cite{g} and dual Cauchy formulas \cite{n}.
{Bachmann-Charlton \cite{bc} defined so-called diagonal 10th variation Schur functions, and generalized the Jacobi-Trudi formulas for the skew type Schur multiple zeta functions shown in \cite{npy} and answered questions about the Checkerboard style Schur multiple zeta values posed in \cite{by}.}
On the other hand, it remains the questions whether these explicit formulas without assumptions and other well-known formulas for Schur polynomials can be obtained.

Among these questions, in this article we will discuss the {\it Pieri formulas} for Schur multiple zeta functions.
The well-known original Pieri formulas for Schur polynomials are the same as formulas found by Pieri for multiplying Schubert varieties in the intersection (cohomology) ring of a Grassmannian, and
these are also known as the special cases of Littlewood-Richardson rule (see, \cite{fu} or \cite{ma}).
Let us explain it breafly.
Let $s_{\lambda}(x_1, \cdots, x_m)$ be a Schur polynomial for certain semi-standard Young tableaux of shape $\lambda$.
{For the Young diagrams of shapes $(r)$ and $(\{1\}^r)$ with one row and one column of length $r\in \mathbb N$, then the Pieri formulas} are
 \[
     %\label{col}
      s_{\lambda}(x_1, \ldots, x_m)\cdot s_{(r)}(x_1, \ldots, x_m)=\sum_{\mu} s_{\mu}(x_1, \ldots, x_m),
 \]
 where the sum is taken over all $\mu$'s that are obtained from $\lambda$ by adding $r$ boxes, with no two in the same column, and
  \[
     %\label{row}
      s_{\lambda}(x_1, \ldots, x_m)\cdot s_{(\{1\}^r)}(x_1, \ldots, x_m) = \sum_{\mu} s_{\mu}(x_1, \ldots, x_m),
 \]
 where the sum is taken over all $\mu$'s that are obtained from $\lambda$ by adding $r$ boxes, with no two in the same row.
 %%%%%%
 %%%%%%
 %%%%%%
 
Considering those formulas for Schur multiple zeta functions, we first have to fix the positions for adding $r$ boxes.
It may be natural to put those boxes to the most-right boxes in each row or  down to the most-bottom box in each column.
But here comes a new problem that which box go where. This is important because it affects the result since the value of the Schur multiple zeta function changes depending on 
the location of the variables.
Toward this problem, we introduce a new insertion method called {\it Pushing rule}.
We will give a detailed explanation of that rule in Section \ref{push}, and  explain it briefly here.
Put the boxes on the top of some columns of $\pmb s$ (or to the left of some rows of $\pmb s$) and push the columns (or rows) from top (or left), respectively.
Following example is for $\lambda=(2,1)$ and $r=1$.
We put 
\ytableausetup{boxsize=normal,aligntableaux=center}
 \begin{ytableau}
  t
\end{ytableau}
to the top of the column of $\pmb s$ and push that row from the top.
{
  \begin{align*}
      &\zeta_{(2,1)}\left(
 \ytableausetup{boxsize=normal}
  \begin{ytableau}
  s_{11} & s_{12} \\
  s_{21}
  \end{ytableau}\right)
\cdot \zeta_{(1)}\left(
 \ytableausetup{boxsize=normal}
  \begin{ytableau}
t  \end{ytableau}
\right)\\
&=\zeta_{(2,1,1)}\left(
 \ytableausetup{boxsize=normal}
  \begin{ytableau}
  t & s_{12} \\
  s_{11}\\
  s_{21}
  \end{ytableau}\right)
  +\zeta_{(2,2)}\left(
 \ytableausetup{boxsize=normal}
  \begin{ytableau}
  s_{11} & t \\
s_{21}&s_{12}
  \end{ytableau}\right)+\zeta_{(3,1)}\left(
 \ytableausetup{boxsize=normal}
  \begin{ytableau}
  s_{11} & s_{12} &t\\
   s_{21}
  \end{ytableau}\right)+({\rm{error\; terms}}).
  \end{align*}}
 As we can see, even this simple example includes the non-zero error terms.
 The problem is how to handle them. We try to avoid this difficulty by taking the summation 
 of the products $\zeta_{\lambda} \cdot \zeta_{(r)}$ or $\zeta_{\lambda} \cdot \zeta_{(\{1\}^r)}$ with swapped variables, 
 and obtain the following 
 
\begin{theorem}
\label{maintheorem}
For any positive integers $\ell,m$ and $k$, we assume that $\Re(y_i)>1\ (1\le i\le \ell)$ and $\Re(z_j)>1\ (1\le j\le \ell-1)$. Then it holds that
\[
\ytableausetup{boxsize=1.5em}
\sum_{sym}\zeta_{({\ell},\{1\}^{k})}\left(\begin{ytableau}
  y_1&\cdots&y_{{\ell}}\\
  x_1\\
  \vdots\\
  x_k
\end{ytableau}\right)\cdot\zeta_{(m)}\left(\begin{ytableau}
  z_1&z_2&\cdots&z_m
\end{ytableau}\right)=\sum_{sym}\sum_{{\pmb u}_{\mu} \in {\pmb U}_H}\zeta_{\mu}({\pmb u}_{\mu}),
\]
where $\displaystyle{\sum_{sym}}$ means the summation over the permutation of $S=\{y_{1},\ldots,y_{{\ell}},z_1,\ldots,z_{{\ell}-1}\}$ as indeterminates and the inner sum in the right-hand side is taken all the terms ${\pmb u}_{\mu}\in {\pmb U}_H$ obtained by the pushing rule and $\mu$ is the shape of ${\pmb u}_{\mu}$.
Also, we assume that $\Re(x_i)>1\ (1\le i\le k)$ and $\Re(z_j)>1\ (1\le j\le k-1)$ then we have
\[
\ytableausetup{boxsize=1.5em}
\sum_{sym}\zeta_{(\ell+1,\{1\}^{k-1})}\left(\begin{ytableau}
  x_1&y_1&\cdots&y_{{\ell}}\\
  x_2\\
  \vdots\\
  x_k
\end{ytableau}\right)\cdot\zeta_{(\{1\}^{m})}\left(\begin{ytableau}
  z_1\\
  z_2\\
  \vdots\\
  z_m
\end{ytableau}\right)=\sum_{sym}\sum_{{\pmb u}_{\mu}\in {\pmb U}_E}\zeta_{\mu}({\pmb u}_{\mu}),
\]
where $\displaystyle{\sum_{sym}}$ are taken all over the permutation of $\{x_{1},\ldots,x_{k},z_1,\ldots,z_{k-1}\}$ as indeterminates and the inner sum in the right-hand side is taken all the terms ${\pmb u}_{\mu}\in {\pmb U}_E$ obtained by the pushing rule.
\end{theorem}
  %%%%%%
 %%%%%%

 Theorem \ref{maintheorem} follows from the following extended Jacobi-Trudi formula, which states some kind of generalization of the Jacobi-Trudi formula for Schur multiple zeta functions (\cite{npy}):
 \begin{theorem}\label{extendJT}
With any integers $m\ge n\ge 1$, for $\Re(s_{1m}),\Re(s_{1(n-1)}),\Re(s_{2n})>1$, it holds that 
\begin{align*}
  \sum_{diag}\zeta_{(m,n)}\left(\begin{ytableau}
  s_{11}&s_{12}&\cdots&\cdots&\cdots&s_{1m}\\
  s_{21}&s_{22}&\cdots&s_{2n}\\
\end{ytableau}\right)
&=\sum_{diag}\begin{vmatrix}
\zetas(s_{11},\ldots,s_{1m})&\zetas(s_{21},\ldots,s_{2n},s_{1n}, \ldots,s_{1m})\\
\zetas(s_{11},\ldots,s_{1(n-1)})&\zetas(s_{21},\ldots,s_{2n})
\end{vmatrix},
\end{align*}
where the sum \[\sum_{diag}=\sum_{\substack{\sigma_j\in S_j\\j\in J}}\prod_{i=1}^{n-1}\sigma_i\]
for $J=\{1,\ldots,n-1\}$ and the sets of permutation $S_j=\{{\rm id},(1j,2(j+1))\}$.
 \end{theorem}
%We also give more general forms of Theorem \ref{extendJT} in this paper, but we need more definitions to explain them. Therefore, we explain the simplest form as an example.
%We note that we use these general forms (Theorem \ref{mnx}) to obtain Theorem \ref{maintheorem}.
{ In this article, first,} we will study the extended Jacobi-Trudi formula in Section \ref{jacobitrudi} after stating the basic terminology in Section \ref{prelim}. Next, in the Section \ref{pieri}, we will give the proof of Theorem \ref{maintheorem} and examples. In the final section, we will discuss further applications.
We mainly deal with the Pieri formula for the hook type Schur multiple zeta functions {in this article, because}
as we observe in the final section, it is complicated to consider the Pieri formula for non-hook type Schur multiple zeta functions, since they have too many error terms.

\section{Preliminaries}\label{prelim}
We review combinatorial settings of Schur multiple zeta functions of \cite{npy}.
{Especially,} in this section, we consider the truncated Schur multiple zeta functions, that is, we use the following truncated sum:
 For $N\in \mathbb{N}$, let $\mathrm{SSYT}_N(\lambda)$ be the set of all $(m_{ij})\in \mathrm{SSYT}(\lambda)$ such that $m_{ij}\le N$ for all $i,j$.
 Define
\[
 \zeta^{N}_{\lambda}({\pmb s})
=\sum_{M\in\mathrm{SSYT}_N(\lambda)}\frac{1}{M^{{\pmb s}}}.
\]
 Notice that
 $\displaystyle{\lim_{N \to \infty} \zeta^N_{\lambda}({\pmb s})=\zeta_{\lambda}({\pmb s})}$ when ${\pmb s}\in W_\lambda$.
\subsection{Rim decomposition of partition}
 A skew Young diagram $\theta$ is a diagram obtained as a set difference of two Young diagrams of partitions $\lambda$ and $\mu$
 satisfying $\mu\subset \lambda$, that is $\mu_i\le \lambda_i$ for all $i$.
 In this case, we write $\theta=\lambda/\mu$.
 It is called a {\it ribbon} if it is connected and contains no $2\times 2$ block of boxes.
 Let $\lambda$ be a partition.
 The maximal outer ribbon of $\lambda$ is called the {\it rim of} $\lambda$.
 We can peel the diagram $\lambda$ off into successive rims $\theta_t,\theta_{t-1},\ldots,\theta_1$ beginning from the outside of $\lambda$.
 We call $\Theta=(\theta_1,\ldots,\theta_t)$ a {\it rim decomposition of $\lambda$}.
 In other words, we consider a sequence of Young diagrams $\emptyset=\lambda^{(0)},\lambda^{(1)},\ldots,\lambda^{(t)}=\lambda$
 such that $\lambda^{(i-1)}\subset\lambda^{(i)}$ and $\lambda^{(i)}/\lambda^{(i-1)}$ is the ribbon $\theta_i$ for all $1\le i\le t$.
 
\begin{exam}
\label{ex:rim}
 The following $\Theta=(\theta_1,\theta_2,\theta_3,\theta_4)$ is a rim decomposition of $\lambda=(4,3,3,2)$;
\[
\ytableausetup{boxsize=normal,aligntableaux=center}
 \Theta
=\,
\begin{ytableau}
 1 & 1 & 3 & 3 \\
 2 & 3 & 3 \\
 2 & 3 & 4 \\
 3 & 3 
\end{ytableau}
,
\]
 which means that \!\!\!\!\!
\ytableausetup{boxsize=10pt,aligntableaux=center}
%matheoremode,をけしてる.
 $\theta_1=\ydiagram{2}$\,,
 $\theta_2=\ydiagram{0,1,1}$\,,
 $\theta_3=\ydiagram{2+2,1+2,1+1,2}$\, and
 $\theta_4=\ydiagram{1}$\,.
\end{exam}

 Write $\lambda=(\lambda_1,\ldots,\lambda_r)$.
 We call a rim decomposition $\Theta=(\theta_1,\ldots,\theta_r)$ of $\lambda$ an {\it $H$-rim decomposition} 
 if each $\theta_i$ starts from $(i,1)$ for all $1\le i\le r$.
 Here, we permit $\theta_i=\emptyset$. 
 We denote $\mathrm{Rim}^{\lambda}_H$ by the set of all $H$-rim decompositions of $\lambda$.

\begin{exam}
\label{ex:Hrim}
 The following $\Theta=(\theta_1,\theta_2,\theta_3,\theta_4)$ is an $H$-rim decomposition of $\lambda=(4,3,3,2)$; 
\[
\ytableausetup{boxsize=normal,aligntableaux=center}
 \Theta
=\,
\begin{ytableau}
 1 & 1 & 3 & 3 \\
 3 & 3 & 3 \\
 3 & 4 & 4 \\
 4 & 4 
\end{ytableau}
,
\]
 which means that \!\!\!\!\!
\ytableausetup{boxsize=10pt,aligntableaux=center} 
%matheoremode,をけしてる.
 $\theta_1=\ydiagram{2}$\,,
 $\theta_2=\emptyset$\,, 
 $\theta_3=\ydiagram{2+2,3,1}$\, and 
 $\theta_4=\ydiagram{1+2,2}$\,.
 Note that the rim decomposition appeared in Example~\ref{ex:rim} is not an $H$-rim decomposition.
\end{exam}
Also, a rim decomposition $\Theta=(\theta_1,\ldots,\theta_s)$ of $\lambda$ is called
 an {\it $E$-rim decomposition} if each $\theta_i$ starts from $(1,i)$ for all $1 \le i \le s$. 
 Here, we again permit $\theta_i=\emptyset$.
 We denote by $\mathrm{Rim}^{\lambda}_E$ the set of all $E$-rim decompositions of $\lambda$.

\subsection{Patterns on the $\mathbb{Z}^2$ lattice}

 Fix $N\in\mathbb{N}$.
 For a partition $\lambda=(\lambda_1,\ldots,\lambda_r)$,
 let $a_i$ and $b_i$ be lattice points in $\mathbb{Z}^2$
 respectively given by $a_i=(r+1-i,1)$ and $b_i=(r+1-i+\lambda_i,N)$ for $1\le i\le r$.
 Put $A=(a_1,\ldots,a_r)$ and $B=(b_1,\ldots,b_r)$.
 An {\it $H$-pattern} corresponding to $\lambda$ is a tuple $L=(l_1,\ldots,l_r)$ of directed paths on $\mathbb{Z}^2$,
 whose directions are allowed only to go one to the right or one up,
 such that $l_i$ starts from $a_i$ and ends to $b_{\sigma(i)}$ for some permutation $\sigma\in \mathfrak S_r$, where $\mathfrak S_r$ is the Symmetric group of $r$ elements.
 We call such $\sigma\in\mathfrak  S_r$ the {\it type} of $L$ and denote it by $\sigma=\mathrm{type}(L)$.
 Note that the type of an $H$-pattern does not depend on $N$.
 Let $\mathcal{H}^{N}_{\lambda}$ be the set of all $H$-patterns corresponding to $\lambda$.
% $S^{\lambda}_H=\{\mathrm{type}(L)\in S_r\,|\,L\in \mathcal{H}^{N}_{\lambda}\}$.
% The following is a key lemma of our study, which is easily verified.
 
%\begin{lemma}
%\label{lem:rimtypeH}
 %For $\Theta=(\theta_1,\ldots,\theta_r) \in \mathrm{Rim}^{\lambda}_H$, there exists $L=(l_1,\ldots,l_r)\in \mathcal{H}^{N}_{\lambda}$  such that $\mathrm{hd}(l_i)=|\theta_i|$ for all $1\le i\le r$. Moreover, the map $\tau_H:\mathrm{Rim}^{\lambda}_H\to S^{\lambda}_H$ given by $\tau_H(\Theta)=\mathrm{type}(L)$ is a bijection.
%\end{lemma}

%\begin{exam}
% Let $\lambda=(4,3,3,2)$. Then, we have $\tau_H(\Theta)=(1243)\in S_4$  where $\Theta$ is the $H$-rim decomposition of $\lambda$ appeared in Example~\ref{ex:Hrim}.\end{exam}
We can also consider similar argument for $E$-rim.
 Let $c_i$ and $d_i$ be lattice points in $\mathbb{Z}^2$
 respectively given by $c_i=(s+1-i,1)$ and $d_i=(s+1-i+\lambda'_i,N+1)$ for $1\le i\le s$.
 Put $C=(c_1,\ldots,c_s)$ and $D=(d_1,\ldots,d_s)$.
 An {\it $E$-pattern} corresponding to $\lambda$ is a tuple $L=(l_1,\ldots,l_s)$ of directed paths on $\mathbb{Z}^2$,
 whose directions are allowed only to go one to the northeast or one up,
 such that $l_i$ starts from $c_i$ and ends to $d_{\sigma(i)}$ for some $\sigma\in \mathfrak S_s$.
 We also call such $\sigma\in\mathfrak S_s$ the {\it type} of $L$ and denote it by $\sigma=\mathrm{type}(L)$.
Let $\mathcal{E}^{N}_{\lambda}$ be the set of all $E$-patterns corresponding to $\lambda$.
%and $S^{\lambda}_E=\{\mathrm{type}(L)\in S_s\,|\,L\in \mathcal{E}^{N}_{\lambda}\}$.
\subsection{Weight of patterns}

 Fix ${\pmb s}=(s_{ij})\in W_\lambda$.
 We next assign a weight to $L=(l_1,\ldots,l_r) \in \mathcal{H}^{N}_{\lambda}$ via the $H$-rim decomposition of $\lambda$ as follows.
 Take $\Theta=(\theta_1,\ldots,\theta_r)\in\mathrm{Rim}^{\lambda}_{H}$ such that $\tau_H(\Theta)=\mathrm{type}(L)$.
 Then, when the $k$-th horizontal edge of $l_i$ is on the $j$-th row,
 we weight it with $\frac{1}{j^{s_{pq}}}$ where $(p,q)\in D(\lambda)$ is the $k$-th component of $\theta_i$.
 Now, the weight $w^N_{{\pmb s}}(l_i)$ of the path $l_i$ is defined to be the product of weights of all horizontal edges along $l_i$.
 Here, we understand that $w^{N}_{{\pmb s}}(l_i)=1$ if $\theta_i=\emptyset$.
 Moreover, we define the weight $w^N_{{\pmb s}}(L)$ of $L\in \mathcal{H}^{N}_{\lambda}$ by 
\[
 w^N_{{\pmb s}}(L)=\prod^{r}_{i=1}w^{N}_{{\pmb s}}(l_i).
\]

\begin{exam}
 Let $\lambda=(4,3,3,2)$.
 Consider the following $L=(l_1,l_2,l_3,l_4)\in \mathcal{H}^{4}_{(4,3,3,2)}$;
\begin{figure}[ht]
\begin{center}
%{ PICTURE}
% \includegraphics[clip,width=85mm]{Hpattern.pdf}
 \begin{tikzpicture} 
  \node at (0.5,1) {$1$};
  \node at (0.5,2) {$2$};
  \node at (0.5,3) {$3$};
  \node at (0.5,4) {$4$};
     \node at (1,0) {$a_4$};  
     \node at (2,0) {$a_3$};
     \node at (3,0) {$a_2$};
     \node at (4,0) {$a_1$};
     \node at (1,0.5) {$(1,1)$};  
     \node at (2,0.5) {$(2,1)$};
     \node at (3,0.5) {$(3,1)$};
     \node at (4,0.5) {$(4,1)$};
       \node at (3,4.5) {$(3,4)$};
       \node at (5,4.5) {$(5,4)$};
       \node at (6,4.5) {$(6,4)$}; 
       \node at (8,4.5) {$(8,4)$};
       \node at (3,5) {$b_4$};
       \node at (5,5) {$b_3$};
       \node at (6,5) {$b_2$}; 
       \node at (8,5) {$b_1$}; 
   \node at (1,1) {$\bullet$};
   \node at (1,2) {$\bullet$};
   \node at (1,3) {$\bullet$};
   \node at (1,4) {$\bullet$};    
   \node at (2,1) {$\bullet$};
   \node at (2,2) {$\bullet$};
   \node at (2,3) {$\bullet$};  
   \node at (2,4) {$\bullet$};  
   \node at (3,1) {$\bullet$};
   \node at (3,2) {$\bullet$};
   \node at (3,3) {$\bullet$};  
  \node at (3,4) {$\bullet$};  
   \node at (4,1) {$\bullet$};
   \node at (4,2) {$\bullet$};
   \node at (4,3) {$\bullet$};  
   \node at (4,4) {$\bullet$};  
   \node at (5,1) {$\bullet$};
   \node at (5,2) {$\bullet$};
   \node at (5,3) {$\bullet$};  
   \node at (5,4) {$\bullet$};  
   \node at (6,1) {$\bullet$};
   \node at (6,2) {$\bullet$};
   \node at (6,3) {$\bullet$}; 
   \node at (6,4) {$\bullet$};  
   \node at (7,1) {$\bullet$};
   \node at (7,2) {$\bullet$};
   \node at (7,3) {$\bullet$}; 
   \node at (7,4) {$\bullet$};  
   \node at (8,1) {$\bullet$};
   \node at (8,2) {$\bullet$};
   \node at (8,3) {$\bullet$};  
   \node at (8,4) {$\bullet$};  
 \draw (4,1) -- (5,1) -- (5,2) -- (6,2) -- (6,3) -- (6,4);
 \draw[dotted] (3,1) -- (3,2) -- (3,3) -- (3,4);
 \draw[loosely dashdotted] (2,1) -- (2,2) -- (2,3) -- (3,3) -- (4,3) -- (5,3) -- (6,3) -- (7,3) -- (7,4) -- (8,4);      
 \draw[dashed] (1,1) -- (1,2) -- (2,2) -- (3,2) -- (4,2) -- (5,2) -- (5,3) -- (5,4);
 \end{tikzpicture}
\end{center}
\ \\[-10pt]
\caption{$L=(l_1,l_2,l_3,l_4)\in \mathcal{H}^{4}_{(4,3,3,2)}$}
\end{figure}

 Since $\mathrm{type}(L)=(1243)$,
 the corresponding $H$-rim decomposition of $\lambda$ is nothing but the one
 appeared in Example~\ref{ex:Hrim}. \\[-5pt]
 
 Let
$
\ytableausetup{boxsize=18pt,aligntableaux=center}
{\pmb s}=\,
\begin{ytableau}
 a & b & c & d \\
 e & f & g \\
 h & i & j \\
 k & l 
\end{ytableau}
\in T((4,3,3,2),\mathbb{C})
$. 
 Then, the weight of $l_i$ are given by  
\[
 w^{4}_{{\pmb s}}(l_1)
=\frac{1}{1^a2^b}, \quad
 w^{4}_{{\pmb s}}(l_2)
=1, \quad
 w^{4}_{{\pmb s}}(l_3)
=\frac{1}{3^h3^e3^f3^g3^c4^d}, \quad
 w^{4}_{{\pmb s}}(l_4)
=\frac{1}{2^k2^l2^i2^j}.
\]
\end{exam}
{ Similarly, we also define a weight on $L=(l_1,\ldots,l_s) \in \mathcal{E}^{N}_{\lambda}$ via the $E$-rim decomposition of $\lambda$ as follows.
 Take $\Theta=(\theta_1,\ldots,\theta_s)\in\mathrm{Rim}^{\lambda}_{E}$ such that $\tau_E(\Theta)=\mathrm{type}(L)$.
 Then, when the $k$-th northeast edge of $l_i$ lies from the $j$-th row to $(j+1)$-th row,
 we weight it with $\frac{1}{j^{s_{pq}}}$ where $(p,q)\in D(\lambda)$ is the $k$th component of $\theta_i$.
 Now, the weight $w^N_{{\pmb s}}(l_i)$ of the path $l_i$ is defined to be the product of weights of all northeast edges along $l_i$.
 Here, we understand that $w^{N}_{{\pmb s}}(l_i)=1$ if $\theta_i=\emptyset$.
 Moreover, we define the weight $w^N_{{\pmb s}}(L)$ of $L\in \mathcal{E}^{N}_{\lambda}$ by 
\[
 w^N_{{\pmb s}}(L)=\prod^{s}_{i=1}w^{N}_{{\pmb s}}(l_i).
\]

\begin{exam}
 Let $\lambda=(4,3,3,2)$.
 Consider the following $L=(l_1,l_2,l_3,l_4)\in \mathcal{E}^{6}_{(4,3,3,2)}$;
\begin{figure}[h]
\begin{center}
 \begin{tikzpicture} 
  \node at (0.5,1) {$1$};
  \node at (0.5,2) {$2$};
  \node at (0.5,3) {$3$};
  \node at (0.5,4) {$4$};
  \node at (0.5,5) {$5$};
  \node at (0.5,6) {$6$};
  \node at (0.5,7) {$7$};
     \node at (1,0) {$c_4$};  
     \node at (2,0) {$c_3$};
     \node at (3,0) {$c_2$};
     \node at (4,0) {$c_1$};
     \node at (1,0.5) {$(1,1)$};  
     \node at (2,0.5) {$(2,1)$};
     \node at (3,0.5) {$(3,1)$};
     \node at (4,0.5) {$(4,1)$};
       \node at (2,7.5) {$(2,7)$};
       \node at (5,7.5) {$(5,7)$};
       \node at (7,7.5) {$(7,7)$}; 
       \node at (8,7.5) {$(8,7)$};
       \node at (2,8) {$d_4$};
       \node at (5,8) {$d_3$};
       \node at (7,8) {$d_2$}; 
       \node at (8,8) {$d_1$}; 
   \node at (1,1) {$\bullet$};
   \node at (1,2) {$\bullet$};
   \node at (1,3) {$\bullet$};
   \node at (1,4) {$\bullet$};   
   \node at (1,5) {$\bullet$};
   \node at (1,6) {$\bullet$};
   \node at (1,7) {$\bullet$};    
   \node at (2,1) {$\bullet$};
   \node at (2,2) {$\bullet$};
   \node at (2,3) {$\bullet$};  
   \node at (2,4) {$\bullet$};
   \node at (2,5) {$\bullet$};
   \node at (2,6) {$\bullet$};
   \node at (2,7) {$\bullet$};    
   \node at (3,1) {$\bullet$};
   \node at (3,2) {$\bullet$};
   \node at (3,3) {$\bullet$};  
   \node at (3,4) {$\bullet$};
   \node at (3,5) {$\bullet$};
   \node at (3,6) {$\bullet$};
   \node at (3,7) {$\bullet$};    
   \node at (4,1) {$\bullet$};
   \node at (4,2) {$\bullet$};
   \node at (4,3) {$\bullet$};  
   \node at (4,4) {$\bullet$};  
   \node at (4,5) {$\bullet$};
   \node at (4,6) {$\bullet$};
   \node at (4,7) {$\bullet$};    
   \node at (5,1) {$\bullet$};
   \node at (5,2) {$\bullet$};
   \node at (5,3) {$\bullet$};  
   \node at (5,4) {$\bullet$};  
   \node at (5,5) {$\bullet$};
   \node at (5,6) {$\bullet$};
   \node at (5,7) {$\bullet$};    
   \node at (6,1) {$\bullet$};
   \node at (6,2) {$\bullet$};
   \node at (6,3) {$\bullet$}; 
   \node at (6,4) {$\bullet$};  
   \node at (6,5) {$\bullet$};
   \node at (6,6) {$\bullet$};
   \node at (6,7) {$\bullet$};    
   \node at (7,1) {$\bullet$};
   \node at (7,2) {$\bullet$};
   \node at (7,3) {$\bullet$}; 
   \node at (7,4) {$\bullet$};  
   \node at (7,5) {$\bullet$};
   \node at (7,6) {$\bullet$};
   \node at (7,7) {$\bullet$};    
   \node at (8,1) {$\bullet$};
   \node at (8,2) {$\bullet$};
   \node at (8,3) {$\bullet$};  
   \node at (8,4) {$\bullet$};  
   \node at (8,5) {$\bullet$};
   \node at (8,6) {$\bullet$};
   \node at (8,7) {$\bullet$};    
 \draw (4,1) -- (5,2) -- (5,3) -- (5,4) -- (5,5) -- (6,6) -- (7,7);
 \draw[dotted] (3,1) -- (3,2) -- (3,3) -- (4,4) -- (4,5) -- (5,6) -- (5,7);
 \draw[loosely dashdotted] (2,1) -- (3,2) -- (4,3) -- (5,4) -- (6,5) -- (7,6) -- (8,7);      
 \draw[dashed] (1,1) -- (1,2) -- (1,3) -- (2,4) -- (2,5) -- (2,6) -- (2,7);
 \end{tikzpicture}
\end{center}
\ \\[-10pt]
\caption{$L=(l_1,l_2,l_3,l_4)\in \mathcal{E}^{6}_{(4,3,3,2)}$}
\end{figure}

 Let
$
\ytableausetup{boxsize=18pt,aligntableaux=center}
{\pmb s}=\,
\begin{ytableau}
 a & b & c & d \\
 e & f & g \\
 h & i & j \\
 k & l 
\end{ytableau}
\in T((4,3,3,2),\mathbb{C})
$. 
 Then, the weight of $l_i$ are given by  
\[
 w^{4}_{{\pmb s}}(l_1)
=\frac{1}{1^a5^e6^h}, \quad
 w^{4}_{{\pmb s}}(l_2)
=\frac{1}{3^b5^f}, \quad
 w^{4}_{{\pmb s}}(l_3)
=\frac{1}{1^c2^g3^j4^i5^l6^k}, \quad
 w^{4}_{{\pmb s}}(l_4)
=\frac{1}{3^d}.
\]
\end{exam}}
{
\subsection{Jacobi-Trudi formulas}
The Jacobi-Trudi formula expresses the Schur polynomial as a determinant in terms of the complete symmetric polynomials $h_i(=s_{(r)})$ or the elementary symmetric polynomials $e_i(=s_{\{1\}^r})$. Let $\lambda=(\lambda_1,\ldots,\lambda_r)$ be a partition and $\lambda'=(\lambda_1',\ldots,\lambda_s')$ be the conjugate of $\lambda$ {defined by $\lambda_i'= \#\{j ~|~ \lambda_j \ge i\}$.} Then, the Jacobi-Trudi formula are
\[
s_\lambda(x_1,\ldots,x_n)={\rm det}[h_{(\lambda_i+j-i)}(x_1,\ldots,x_n)]_{1\le i,j\le r},
\]
and
\[
s_\lambda(x_1,\ldots,x_n)={\rm det}[e_{(\lambda_i'+j-i)}(x_1,\ldots,x_n)]_{1\le i,j\le s}.
\]
Taking into account formula (\ref{zetazetas}), we can expect that the Schur multiple zeta functions are also expressed as a determinant in terms of the multiple zeta star functions $\zetas$ or the multiple zeta functions $\zeta$. In fact, Nakasuji-Phuksuwan-Yamasaki showed the following Jacobi-Trudi formulas for SMZFs \cite{npy} by using the Gessel–Viennot method \cite{gv}:
\begin{prop}[Nakasuji-Phuksuwan-Yamasaki. 2018 \cite{npy}]
\label{npyjt}
  Let $\lambda=(\lambda_1,\ldots,\lambda_r)$ be a partition and $\lambda'=(\lambda_1',\ldots,\lambda_s')$ the conjugate of $\lambda$. Also, we assume ${\pmb s}=(s_{ij})=(a_{j-i})\in W_\lambda$.
 \begin{enumerate}
     \item Assume that $\Re(s_{i,\lambda_i})>1$ for all $1\le i \le r$. Then we have
\[
 \zeta_{\lambda}({\pmb s})
=\det\left[\zeta^{\star}(a_{-j+1},a_{-j+2},\ldots,a_{-j+(\lambda_{i}-i+j)})\right]_{1\le i,j\le r\,}.
\]
 Here, we understand that $\zeta^{\star}(\,\cdots)=1$ if $\lambda_{i}-i+j=0$ and $0$ if $\lambda_{i}-i+j<0$.
 \item Assume that $\Re(s_{\lambda_i',i})>1$ for all $1\le i \le s$. Then we have
\[
 \zeta_{\lambda}({\pmb s})
=\det\left[\zeta(a_{j-1},a_{j-2},\ldots,a_{j-(\lambda_{i}'-i+j)})\right]_{1\le i,j\le s\,}.
\]
 Here, we understand that $\zeta(\,\cdots)=1$ if $\lambda_{i}'-i+j=0$ and $0$ if $\lambda_{i}'-i+j<0$.
 \end{enumerate}
\end{prop}
We explain the key points of the proof of (1).
 Let $\mathcal{H}^{N}_{\lambda,0}$ be the set of all $L=(l_1,\ldots,l_r)\in \mathcal{H}^{N}_{\lambda}$ 
 such that any distinct pair of $l_i$ and $l_j$ has no intersection and define 
  \[X^N_{\lambda}(\pmb s)=\sum_{\sigma\in S_H^\lambda}\varepsilon_\sigma\prod_{i=1}^r{\zetas}^N(\theta_i^\sigma(\pmb{s})),\ X^N_{\lambda,1}(\pmb s)=\sum_{L\in H^N_{\lambda}\setminus H^N_{\lambda,0}}\varepsilon_{type(L)}w^N_{\pmb s}(L),\]
where $\varepsilon_\sigma$ is the signature of $\sigma\in S_r$.
Then, they showed that 
\begin{equation}
\label{npykey}
\zeta_{\lambda}^N(\pmb s)=\sum_{\sigma\in S_H^\lambda}\varepsilon_\sigma\prod_{i=1}^r{\zetas}^N(\theta_i^\sigma(\pmb{s}))-\sum_{L\in H^N_{\lambda}\setminus H^N_{\lambda,0}}\varepsilon_{type(L)}w^N_{\pmb s}(L),
\end{equation}
Moreover, if  ${\pmb s}=(s_{ij})=(a_{j-i})\in W_\lambda$, then we obtain
\[
X^N_{\lambda}(\pmb s)=\det
\left[ {\zetas}^N(a_{-j+1},a_{-j+2},\ldots,a_{-j+(\lambda_i-i+j)})\right]_{1\le i,j\le r},\ X^N_{\lambda,1}({\pmb s})=0.
\]
This leads to (1). {Also, (2) can be shown by using $E$-rim decomposition and similar argument.} They also gave the following remark:
\begin{rmk}
\label{rmk:ErrorTerms}
 In some cases, 
 $X^{N}_{\lambda}({\pmb s})$ actually has a determinant expression without the assumption on variables;
\begin{align*}
\ytableausetup{boxsize=normal,aligntableaux=center}
 X^{N}_{(2,2)}
\left(~
\begin{ytableau}
 a & b \\
 c & d
\end{ytableau} 
~\right) 
&=\left|
\begin{array}{cc}
{\zetas}^N(a,b) & {\zetas}^N(c,d,b) \\
 {\zetas}^N(a) & {\zetas}^N(c,d)
\end{array}
\right|,\\[5pt]
 X^{N}_{(2,2,1)}
\left(~
\begin{ytableau}
 a & b \\
 c & d \\
 e
\end{ytableau} 
~\right) 
&=\left|
\begin{array}{ccc}
 {\zetas}^N(a,b) & {\zetas}^N(c,d,b) & {\zetas}^N(e,c,d,b) \\
 {\zetas}^N(a) & {\zetas}^N(c,d) & {\zetas}^N(e,c,d) \\
 0 & 1 & {\zetas}^N(e)
\end{array}
\right|.
\end{align*} 
 However, in general,
 $X^{N}_{\lambda}({\pmb s})$ can not be written as a determinant.
 For example, we have 
\begin{align*}
 X^{N}_{(2,2,2)}
\left(~
\begin{ytableau}
 a & b \\
 c & d \\
 e & f
\end{ytableau} 
~\right) 
&={\zetas}^N(a,b){\zetas}^N(c,d){\zetas}^N(e,f)
-{\zetas}^N(a,b){\zetas}^N(c){\zetas}^N(e,f,d)\\
&\ \ \ -{\zetas}^N(c,a){\zetas}^N(e,f,d,b) -{\zetas}^N(a){\zetas}^N(c,d,b){\zetas}^N(e,f)\\
&\ \ \ +{\zetas}^N(c,a,b){\zetas}^N(e,f,d) +{\zetas}^N(a){\zetas}^N(c){\zetas}^N(e,f,d,b)
\end{align*}  
 and see that the right-hand side does not seem to be expressed as a determinant
 (but is close to the determinant). 
 
 Similarly, $X^{N}_{\lambda,1}({\pmb s})$ does not vanish in general.
 For example,  
\begin{align*}
 X^{2}_{(2,2),1}
\left(~
\begin{ytableau}
 a & b \\
 c & d
\end{ytableau} 
~\right)
%&=\left(\frac{1}{1^a1^b1^c1^d}+\frac{1}{1^a1^b1^c2^d}
%+\frac{1}{1^a2^b1^c1^d}+\frac{1}{1^a2^b1^c2^d}\right.\\
%&\ \ \ \ \ \ \left.\frac{1}{1^a2^b2^c2^d}+\frac{1}{2^a2^b1^c1^d}+\frac{1}{2^a2^b1^c2^d}+\frac{1}{2^a2^b2^c2^d}\right)\\
%&\ \ \ -\left(\frac{1}{1^a1^b1^c1^d}+\frac{1}{1^a2^b1^c1^d}+\frac{1}{1^a2^b1^c2^d}+\frac{1}{1^a2^b2^c2^d}\right.\\
%&\ \ \ \ \ \ \left.+\frac{1}{2^a1^b1^c1^d}+\frac{1}{2^a2^b1^c1^d}+\frac{1}{2^a2^b1^c2^d}+\frac{1}{2^a2^b2^c2^d}\right)\\
&=\frac{1}{1^a1^b1^c2^d}-\frac{1}{2^a1^b1^c1^d},
\end{align*}
 which actually vanishes when $a=d$.
\end{rmk}
Due to Remark \ref{rmk:ErrorTerms}, they assumed that the variables on the same diagonal lines are same. Compared with this, we treat this problem by taking summation over all permutations of component on the same diagonal lines of variables.

}
%%%%%%%%%%%%%%%%%%%%%%%%%%%%%%%%
%%%%%%%%%%%%%%%%%%%%%%%%%%%%%%%%
%%%%%%%%%%%%%%%%%%%%%%%%%%%%%%%%
%%%%%%%%%%    Pushing Rule       %%%%%%%%%%%%
%%%%%%%%%%%%%%%%%%%%%%%%%%%%%%%%
%%%%%%%%%%%%%%%%%%%%%%%%%%%%%%%%
%%%%%%%%%%%%%%%%%%%%%%%%%%%%%%%%

\subsection{Pushing rule}\label{push}
{For ${\pmb s}\in T(\lambda,\mathbb C)$ and ${\pmb t}\in T({(r)},\mathbb C)$ or $T({(\{1\}^r)},\mathbb C)$},
we construct a new Young tableau by inserting all the components in ${\pmb t}$ into ${\pmb s}$.
The insertion method which we use here is called {\bf pushing rule} and the recipe is as follows.

For ${\pmb t} = 
\ytableausetup{boxsize=normal,aligntableaux=center}
 \begin{ytableau}
  t_1 & 
  \cdots &
  t_r
\end{ytableau} \in T((r), {\mathbb C})$, we put $
\ytableausetup{boxsize=normal,aligntableaux=center}
 \begin{ytableau}
  t_1 
\end{ytableau}
$ on the top of some column of $\pmb s$ or next to the 
right-most box in the first row.
For example, when $\lambda=(3,1)$ and ${\pmb s}=
\ytableausetup{boxsize=normal,aligntableaux=center}
 \begin{ytableau}
s_{11} & s_{12} & s_{13} \\
s_{21} 
\end{ytableau}$, then there are 4 choices for $\ytableausetup{boxsize=normal,aligntableaux=center}
 \begin{ytableau}
  t_1 
\end{ytableau}
$:
$$
\ytableausetup{boxsize=normal,aligntableaux=center}
 \begin{ytableau}
*(gray)t_{1}  \\
s_{11} & s_{12} & s_{13}\\
s_{21} 
\end{ytableau}, \hspace{5mm}
\ytableausetup{boxsize=normal,aligntableaux=center}
 \begin{ytableau}
\none &*(gray) t_{1}  \\
s_{11} & s_{12} & s_{13}\\
s_{21} 
\end{ytableau}, \hspace{5mm} 
\ytableausetup{boxsize=normal,aligntableaux=center}
 \begin{ytableau}
\none & \none &*(gray) t_{1}  \\
s_{11} & s_{12} & s_{13}\\
s_{21} 
\end{ytableau}, \hspace{5mm} 
\ytableausetup{boxsize=normal,aligntableaux=center}
 \begin{ytableau}
s_{11} & s_{12}  & s_{13} & *(gray)t_{1} \\
s_{21} 
\end{ytableau}. 
$$
If $\ytableausetup{boxsize=normal,aligntableaux=center}
 \begin{ytableau}
  t_1 
\end{ytableau}
$ is on ${\pmb s}$, then we {\it push} that column from the top to make the shape of the Young tableau. 
If the shape does not become that of the Young tableau after the {\it push}, then we discard that pattern.
Otherwise, we do not do anything more than this:
$$
\ytableausetup{boxsize=normal,aligntableaux=center}
 \begin{ytableau}
*(gray)t_{1} & s_{12} & s_{13}\\
s_{11} \\
s_{21} 
\end{ytableau}, \hspace{5mm}
\ytableausetup{boxsize=normal,aligntableaux=center}
 \begin{ytableau}
s_{11} & *(gray)t_{1} & s_{13}\\
s_{21} & s_{12}
\end{ytableau}, \hspace{5mm} 
{\rm{discard}},\hspace{5mm}
\ytableausetup{boxsize=normal,aligntableaux=center}
 \begin{ytableau}
s_{11} & s_{12}  & s_{13} & *(gray)t_{1} \\
s_{21} 
\end{ytableau}. 
$$
Repeat this procedure to each resulting pattern in the order of 
$\ytableausetup{boxsize=normal,aligntableaux=center}
 \begin{ytableau}
  t_2 
\end{ytableau},
\ytableausetup{boxsize=normal,aligntableaux=center}
 \begin{ytableau}
  t_3
\end{ytableau}
, \cdots, 
\ytableausetup{boxsize=normal,aligntableaux=center}
 \begin{ytableau}
  t_r 
\end{ytableau}
$, where
$
\ytableausetup{boxsize=normal,aligntableaux=center}
 \begin{ytableau}
  t_{i+1} 
\end{ytableau}
$ should be put to the right of $
\ytableausetup{boxsize=normal,aligntableaux=center}
 \begin{ytableau}
  t_{i} 
\end{ytableau}
$ ($1\leq i\leq r-1$) but  not necessarily to the next. Let ${\pmb U}_H$ be the set of new Young tableaux obtained by this {\it pushing rule}.
\begin{exam}
Let $\lambda=(3,2,1)$ and $r=2$.
Then 
\begin{align*}
{\pmb U}_H=
&
\left\{ \hspace{3mm}
\begin{ytableau}
  *(gray)t_1& *(gray)t_2&s_{13}\\
  s_{11}&s_{12}\\
s_{21}&s_{22}\\
s_{31}
\end{ytableau},\hspace{2mm}
\begin{ytableau}
  *(gray)t_1&s_{12}& *(gray)t_2\\
  s_{11}&s_{22}&s_{13}\\
s_{21}\\
s_{31}
\end{ytableau},\hspace{2mm}
\begin{ytableau}
  *(gray)t_1&s_{12}&s_{13}& *(gray)t_2\\
  s_{11}&s_{22}\\
s_{21}\\
s_{31}
\end{ytableau}, \hspace{2mm}
\begin{ytableau}
  s_{11}&*(gray)t_1&*(gray)t_2\\
  s_{21}&s_{12}&s_{13}\\
s_{31}&s_{22}
\end{ytableau}\right.
, \\
& \hspace{2mm}
\left.
\begin{ytableau}
  s_{11}&*(gray)t_1&s_{13}&*(gray)t_2\\
  s_{21}&s_{12}\\
s_{31}&s_{22}
\end{ytableau},
\hspace{2mm} \begin{ytableau}
  s_{11}&s_{12}&*(gray)t_1&*(gray)t_2\\
  s_{21}&s_{22}&s_{13}\\
s_{31}
\end{ytableau},
\hspace{2mm}\begin{ytableau}
  s_{11}&s_{12}&s_{13}&*(gray)t_1&*(gray)t_2\\
  s_{21}&s_{22}\\
s_{31}
\end{ytableau}\hspace{3mm}\right\}.
\end{align*}
\end{exam}

Next, 
for ${\pmb t} = 
\ytableausetup{boxsize=normal,aligntableaux=center}
 \begin{ytableau}
  t_1 \\ 
  \vdots \\
  t_r
\end{ytableau} \in T((\{1\}^r), {\mathbb C})$, we put $
\ytableausetup{boxsize=normal,aligntableaux=center}
 \begin{ytableau}
  t_1 
\end{ytableau}
$ to the left of some row of $\pmb s$ or down to the 
bottommost box in the first column.
If $\ytableausetup{boxsize=normal,aligntableaux=center}
 \begin{ytableau}
  t_1 
\end{ytableau}
$ is to the left of ${\pmb s}$, then we {\it push} that row from the left to make the shape of the Young tableau. 
If the shape does not become that of the Young tableau after the {\it push}, then we discard that pattern.
Otherwise, we do not do anything. 
Repeat this procedure in the order of $\ytableausetup{boxsize=normal,aligntableaux=center}
 \begin{ytableau}
  t_2 
\end{ytableau},
\ytableausetup{boxsize=normal,aligntableaux=center}
 \begin{ytableau}
  t_3
\end{ytableau}
, \cdots, 
\ytableausetup{boxsize=normal,aligntableaux=center}
 \begin{ytableau}
  t_r 
\end{ytableau}
$, where
$
\ytableausetup{boxsize=normal,aligntableaux=center}
 \begin{ytableau}
  t_{i+1} 
\end{ytableau}
$ should be put below $
\ytableausetup{boxsize=normal,aligntableaux=center}
 \begin{ytableau}
  t_i 
\end{ytableau}
$
($1\leq i \leq r-1$) but  not necessarily to the next. Let ${\pmb U}_E$ be the set of new Young tableaux obtained by this.
\begin{exam}
Let $\lambda=(3,2,1)$ and $r=2$.
Then 
\begin{align*}
{\pmb U}_E=
&
\left\{ \hspace{3mm}
\begin{ytableau}
*(gray)t_1&  s_{11}&s_{12}&s_{13}\\
*(gray)t_2& s_{21}&s_{22}\\
s_{31}
\end{ytableau},\hspace{2mm}
\begin{ytableau}
*(gray)t_1&  s_{11}&s_{12}&s_{13}\\
 s_{21}&s_{22}\\
*(gray)t_2&s_{31}
\end{ytableau},\hspace{2mm}
\begin{ytableau}
*(gray)t_1&  s_{11}&s_{12}&s_{13}\\
s_{21}&s_{22}\\
s_{31}\\
*(gray)t_2
\end{ytableau},\right. \hspace{2mm}\\
&
\left.\begin{ytableau}
  s_{11}&s_{12}&s_{13}\\
*(gray)t_1& s_{21}&s_{22}\\
*(gray)t_2&s_{31}
\end{ytableau},\hspace{2mm}
\begin{ytableau}
  s_{11}&s_{12}&s_{13}\\
*(gray)t_1& s_{21}&s_{22}\\
s_{31}\\
*(gray)t_2
\end{ytableau},\hspace{2mm}
\begin{ytableau}
  s_{11}&s_{12}&s_{13}\\
 s_{21}&s_{22}\\
*(gray)t_1&s_{31}\\
*(gray)t_2
\end{ytableau},\hspace{2mm}
\begin{ytableau}
  s_{11}&s_{12}&s_{13}\\
 s_{21}&s_{22}\\
s_{31}\\
*(gray)t_1\\
*(gray)t_2
\end{ytableau}
\hspace{3mm}\right\}.
\end{align*}
\end{exam}

%%%%%%%%%%%%%%%%%%%%%%%%%%%%%%%%
%%%%%%%%%%%%%%%%%%%%%%%%%%%%%%%%
%%%%%%%%%%%%%%%%%%%%%%%%%%%%%%%%
%%%%%%%%%%%%%%%%%%%%%%%%%%%%%%%%
%%%%%%%%%%%%%%%%%%%%%%%%%%%%%%%%

\subsection{Special case}
In this subsection, we focus on the product of type {$\zeta_{(\ell+1,\{1\}^k)}(\pmb s)\cdot\zeta_{(\{1\}^m)}(\pmb t)$ for
% \[\begin{ytableau}
%   x_1&y_1&\cdots&y_{\ell}\\
%   x_2\\
%   \vdots\\
%   x_k\\
% \end{ytableau}\cdot\begin{ytableau}
%   z_1\\
%   \vdots\\
%   z_m
% \end{ytableau}.\]
$\pmb s\in W_{(\ell+1,\{1\}^k)}$ and $\pmb t\in W_{(\{1\}^m)}$.}

First we deal with the case $(k,{\ell},m)=(2,1,2)$. In order to explain the phenomenon, we introduce the following notation for now:
\[\begin{ytableau}
 s_{11}\\
  s_{21}&s_{22}&s_{23}\\
 s_{31}
\end{ytableau}=\sum_{\substack{1\le n_{11}<n_{21}<n_{31}\\n_{21}\le n_{22}\le n_{23}}}\frac{1}{n_{11}^{s_{11}}n_{21}^{s_{21}}n_{22}^{s_{22}}n_{23}^{s_{23}}n_{31}^{s_{31}}}.
\]
Computing directly, we find the following example.
\begin{exam}{For $\begin{ytableau}
  x_1&y_1\\
  x_2\\
\end{ytableau}\in W_{(2,1)}$ and $\begin{ytableau}
  *(gray) z_1\\
  *(gray) z_2
\end{ytableau}\in W_{(1,1)}$, we have}
\begin{align*}
\begin{ytableau}
  x_1&y_1\\
  x_2\\
\end{ytableau}\cdot\begin{ytableau}
  *(gray) z_1\\
  *(gray) z_2
\end{ytableau}&=\begin{ytableau}
  x_1&y_1\\
  x_2\\
  *(gray)z_1\\
  *(gray)z_2
\end{ytableau}+\begin{ytableau}
  x_1&y_1\\
  *(gray)z_1&x_2\\
  *(gray)z_2\\
\end{ytableau}+\begin{ytableau}
  *(gray)z_1&x_1&y_1\\
  x_2\\
  *(gray)z_2
\end{ytableau}+\begin{ytableau}
  *(gray)z_1&x_1&y_1\\
  *(gray)z_2&x_2
\end{ytableau}\\
&-\begin{ytableau}
  *(gray)z_1\\
  x_2&x_1&y_1\\
  *(gray)z_2
\end{ytableau}+\begin{ytableau}
  x_1\\
  *(gray)z_1&x_2&y_1\\
  *(gray)z_2
\end{ytableau}.
\end{align*}
\end{exam}
One can check that if $x_1=x_2=z_1$ then a simple generalization of Pieri formula holds.
In the following, we generalize this result.
For $\lambda=(\lambda_1,\ldots,\lambda_r)$, we define the set $W_{\lambda,R}$ by 
\[
 W_{\lambda,R}
=
\{{\pmb s}=(s_{ij})\in W_\lambda\,|\, \Re(s_{i\lambda_i})>1\text{ for $1\le i\le r$}\}.
\]
\begin{theorem}[$(\lambda=(2,\{1\}^{k-1}))\cdot (\pmb{e}_{k}=(\{1\}^{k}))$]
\label{83}
Let $k$ be a positive integer. Then we have that 
\begin{equation}
    \label{easypieri}
    \ytableausetup{boxsize=1.5em}
\begin{ytableau}
  x_1&y\\
  x_2\\
  \vdots\\
  x_k
\end{ytableau}\cdot\begin{ytableau}
  z_1\\
  z_2\\
  \vdots\\
  z_k
\end{ytableau}=\sum_{\pmb u_\mu\in\pmb U_E}\zeta_{\mu}(\pmb u_\mu)+E_{k}(\pmb{ x};y;\pmb{z}),\ \left(\begin{ytableau}
  x_1&y\\
  x_2\\
  \vdots\\
  x_k
\end{ytableau}\in W_{\lambda,R},\begin{ytableau}
  z_1\\
  z_2\\
  \vdots\\
  z_k
\end{ytableau}\in W_\lambda\right),
\end{equation}
where $\pmb x=(x_{1},\ldots,x_k)$ and $\pmb z=(z_{1},\ldots,z_k)$.
In addition, the error term $E_{k}(\pmb{ x};y;\pmb{z})=0$ if $\pmb x=(a,\ldots,a)$ and $\pmb z=(a,\ldots,a,z_k)$ for $a\in\mathbb{C}$ {with $\Re (a)>1$}.
\end{theorem}
\begin{proof}
We denote by $\pmb{s}_{i_1\ \cdots\ i_k}$ the Young tableau by pushing boxes {$\begin{ytableau}*(gray)z_j\end{ytableau}$} from the left on $i_j$-th row. 
For example,
\[\pmb{s}_{1\ \cdots\ k}=\begin{ytableau}
  *(gray)z_1&x_1&y\\
  *(gray)z_2&x_2\\
  *(gray)\vdots&\vdots\\
  *(gray)z_k&x_k
\end{ytableau},\ \pmb{s}_{2\ \cdots\ k+1}=\begin{ytableau}
  x_1&y\\
  *(gray)z_1&x_2\\
  *(gray)\vdots&\vdots\\
  *(gray)z_{k-1}&x_k\\
  *(gray)z_k
\end{ytableau},\ \pmb{s}_{k+1\ \cdots\ 2k}=\begin{ytableau}
  x_1&y\\
  \vdots\\
  x_k\\
  *(gray)z_1\\
  *(gray)\vdots\\
  *(gray)z_k
\end{ytableau}.\]
We assume $(i_1,\ldots,i_k)\neq (j_1,\ldots, j_k)$ and $i_{\ell}< j_{\ell}$, then we find that $\zeta(\pmb{s}_{i_1\ \cdots\ i_k})$ satisfies the relation $y_{\ell}\le x_{i_{\ell}}$ and $\zeta(\pmb{s}_{j_1\ \cdots\ j_k})$ satisfies the relation $x_{i_{\ell}}<y_{\ell}\le x_{j_{\ell}}$.
Therefore, the right-hand sum of (\ref{easypieri}) does not duplicate the Euler-Zagier multiple zeta functions. 

 If we assume $x_i=z_j=a$ for $i=1,\ldots,k$ and $j=1,\ldots,k-1$, then we confirm that this gives no error terms. 
 By Proposition \ref{npyjt}, the left-hand side of (\ref{easypieri}) is 
 \begin{equation}
 \label{84}
    \zeta(\underbrace{a,\ldots,a}_k)\zeta(y)\zeta(\underbrace{a,\ldots,a}_{k-1},z)-\zeta(y,\underbrace{a,\ldots,a}_k)\zeta(\underbrace{a,\ldots,a}_{k-1},z).
 \end{equation} 
On the other hand, the right-hand side of (\ref{easypieri}) is 
 \begin{equation}
 \label{85}
    \sum_{{\ell}=k}^{2k-1}\begin{vmatrix}
     \zeta(\underbrace{a,\ldots,a}_{{\ell}-1},z)&\zeta(\underbrace{a,\ldots,a}_{{\ell}},z)&\zeta(y,\underbrace{a,\ldots,a}_{{\ell}},z)\\
       \zeta(\underbrace{a,\ldots,a}_{2k-1-{\ell}})&\zeta(\underbrace{a,\ldots,a}_{2k-{\ell}})&\zeta(y,\underbrace{a,\ldots,a}_{2k-{\ell}})\\
       0&1&\zeta(y)
    \end{vmatrix}+\begin{vmatrix}
     \zeta(\underbrace{a,\ldots,a}_{\ell},z)&\zeta(y,\underbrace{a,\ldots,a}_{{\ell}},z)\\
       \zeta(\underbrace{a,\ldots,a}_{2k-1-{\ell}})&\zeta(y,\underbrace{a,\ldots,a}_{2k-1-{\ell}})\\
    \end{vmatrix}.
 \end{equation} 
 Since (\ref{84}) and (\ref{85}) are the same, we complete proving this theorem.
\end{proof}
The next corollary may not be a natural consequence of Theorem \ref{83}, but we can show by the same way to the proof of Theorem \ref{83}. 
\begin{cor}[$(\lambda=({\ell}+1,\{1\}^{k-1}))\cdot (\pmb{e}_{m}=(\{1\}^{m}))$]
\label{216}
For a non-negative integer $\ell$ and positive integers $k$ and $m$, it holds that 
%{if $\Re(x_1)\ge1,\Re(x_i)>1\ (2\le i\le k), \Re(y_j)\ge1\ (1\le j\le \ell-1), \Re(y_\ell)>1,\Re(z_p)\ge1\ (1\le p\le m-1)$ and $\Re(z_m)>1$,}
\[
\ytableausetup{boxsize=1.5em}
\begin{ytableau}
  x_1&y_1&\cdots&y_{{\ell}}\\
  x_2\\
  \vdots\\
  x_k
\end{ytableau}\cdot\begin{ytableau}
  z_1\\
  z_2\\
  \vdots\\
  z_m
\end{ytableau}=\sum_{\pmb u_\mu\in \pmb U_E}\zeta_{\mu}(\pmb u_\mu)+E_{k,{\ell},m}(\pmb{ x};\pmb{y};\pmb{z}),\ \left(\begin{ytableau}
  x_1&y_1&\cdots&y_\ell\\
  x_2\\
  \vdots\\
  x_k
\end{ytableau}\in W_{\lambda,R},\begin{ytableau}
  z_1\\
  z_2\\
  \vdots\\
  z_k
\end{ytableau}\in W_\lambda\right),
\]
where $\pmb x=(x_{1},\ldots,x_k)$, $\pmb y=(y_1,\ldots,y_\ell)$, $\pmb z=(z_{1},\ldots,z_m)$.
In addition, the error term $E_{k,{\ell},m}(\pmb{ x};\pmb{y};\pmb{z})=0$ if $\pmb x=(a,\ldots,a)$ and $\pmb z=(a,\ldots,a,z_k,\ldots,z_m)$ for $a\in\mathbb{C}$ {with $\Re (a)>1$}.
\end{cor}
This corollary ensures that a simple generalization of the Pieri formula holds with some error terms which vanish under some assumptions. In the following sections, taking a special sum of both sides, we remove these assumptions.
Also, we obtain Corollary \ref{216} as one of the corollaries of the theorems shown in Section \ref{pieri}.

%%%%%%%%%%%%sec3%%%%%%%%%%%%%%%%%%%%%%%%%%%%%%%%%%%%%%%%%%%%%%%%%%%%%%%%%%%%%%%%%%%%%%%%%%
%%%%%%%%%%%%%%%%%%%%%%%%%%%%%%%%%%%%%%%%sec3%%%%%%%%%%%%%%%%%%%%%%%%%%%%%%%%%%%%%%%%%%%%%%
%%%%%%%%%%%%%%%%%%%%%%%%%%%%%%%%%%%%%%%%%%%%%%%%%%%%%%%%%%%%%%%%%%%%%sec3%%%%%%%%%%%%%%%%%

\section{Extended Jacobi-Trudi formula}\label{jacobitrudi}
In this section, we consider the extension of the Jacobi-Trudi formula to show the Pieri formulas. This section devotes the proof of Theorem \ref{extendJT}.

First, we show a key lemma for the proof of Theorem \ref{extendJT}. This lemma connects the Schur multiple zeta functions with a sum of products of multiple zeta star functions.
%The proof of this lemma is the same as the proof of Proposition 3.8 of \cite{npy}. Roughly speaking, we show that taking the diagonal sum defined in Theorem \ref{extendJT}, these extra terms vanish.
{In preparation, we define 
\[\sum_{diag}=\sum_{\substack{\sigma_j\in S_j\\j\in \mathbb Z}}\prod_{i\in \mathbb Z}\sigma_i\]
for $S_j$ being the set of permutation of the elements of $I(J)=\{(k,l)\in D(\lambda)~|~l-k=j\}$. We have to note that since the number of boxes in a fixed Young tableau is finite, the product and the sum are finite. Actually, we denote by $\lambda'=(\lambda_1',\ldots,\lambda_s')$ the conjugate of $\lambda$ then we find that
\[\sum_{\substack{\sigma_j\in S_j\\j\in \mathbb Z}}\prod_{i\in \mathbb Z}\sigma_i=\sum_{\substack{\sigma_j\in S_j\\j\in \mathbb Z\cap[2-\lambda_2',\lambda_2-2]}}\prod_{i=2-\lambda_2'}^{\lambda_2-2}\sigma_i.\]
}
{Also, we define a set $W_{\lambda,H}$ by
\[
 W_{\lambda,H}
=
\left\{{\pmb s}=(s_{ij})\in T(\lambda,\mathbb{C})\,\left|\,
\begin{array}{l}
 \text{$\Re(s_{ij})\ge 1$ for all $(i,j)\in D(\lambda) \setminus H(\lambda)$} \\[3pt]
 \text{$\Re(s_{ij})>1$ for all $(i,j)\in H(\lambda)$}
\end{array}
\right.
\right\},
\]
where $H(\lambda)=\{(i,j)\in D(\lambda)~|~i-j\in\{i-\lambda_i~|~1\le i\le r\} \}$.} 
\begin{lemma}
For any partition $\lambda=(\lambda_1,\ldots,\lambda_r)$ and {$\pmb s\in W_{\lambda,H}$}, we have
\label{diag}
\[\sum_{diag}\zeta_\lambda(\pmb{s})=\sum_{diag}\sum_{\sigma\in S_H^\lambda}\varepsilon_\sigma\prod_{i=1}^r{\zetas}(\theta_i^\sigma(\pmb{s})).\]
\end{lemma}

\begin{proof}
{As in \cite{npy}, 
we consider identity (\ref{npykey}):
\[
    \zeta_{\lambda}^N(\pmb s)=\sum_{\sigma\in S_H^\lambda}\varepsilon_\sigma\prod_{i=1}^r{\zetas}^N(\theta_i^\sigma(\pmb{s}))-X_{\lambda,1}^N(\pmb s).
\]
Now we focus on the error $X_{\lambda,1}^N(\pmb s)$. For $L = (l_1,\ldots , l_r) \in H^N_{\lambda}\setminus H^N_{\lambda,0}$ of type $\sigma$, we consider the rightmost intersection point $(p,q)$ appearing in $L$.
For the sake of simplicity, we can assume $l_1$ and $l_2$ cross at $(p,q)$ (FIGURE \ref{fig2}).
Then, we expand $w^N_{\pmb s}(L)$ as \[\prod_{i=1}^{m_1}a_{1i}^{-s_{1i}}\prod_{i=1}^{m_2}a_{2i}^{-s_{2i}}\prod_{j=3}^r\left(\prod_{i=1}^{m_j}a_{ji}^{-s_{ji}}\right).\]
On the other hand, we consider the pair of path $\overline{L}=(\overline l_1,\overline{l_2},l_3,\ldots,l_r)$.
Here, $\overline l_i$ follows $l_i$ until it meets the first intersection point $(p,q)$ and after that follows the other pass $l_j$ to the end (FIGURE \ref{fig3}). 
\begin{figure}[ht]
\begin{tabular}{c}
\begin{minipage}{0.5\hsize}
\begin{center}
%{ PICTURE}
% \includegraphics[clip,width=85mm]{Hpattern.pdf}
 \begin{tikzpicture} 
  \node at (0.5,3) {$q$};
     \node at (1,0.5) {$l_2$};
     \node at (3,0.5) {$l_1$};
     \node at (4,0.5) {$p$};
   \node at (1,1) {$\bullet$};
   \node at (1,2) {$\bullet$};
   \node at (1,3) {$\bullet$};
   \node at (1,4) {$\bullet$};    
   \node at (2,1) {$\bullet$};
   \node at (2,2) {$\bullet$};
   \node at (2,3) {$\bullet$};  
   \node at (2,4) {$\bullet$};  
   \node at (3,1) {$\bullet$};
   \node at (3,2) {$\bullet$};
   \node at (3,3) {$\bullet$};  
  \node at (3,4) {$\bullet$};  
   \node at (4,1) {$\bullet$};
   \node at (4,2) {$\bullet$};
   \node at (4,3) {$\bullet$};  
   \node at (4,4) {$\bullet$};  
   \node at (5,1) {$\bullet$};
   \node at (5,2) {$\bullet$};
   \node at (5,3) {$\bullet$};  
   \node at (5,4) {$\bullet$};  
   \node at (6,1) {$\bullet$};
   \node at (6,2) {$\bullet$};
   \node at (6,3) {$\bullet$};  
   \node at (6,4) {$\bullet$};  
 \draw (3,1) -- (3,2) -- (4,2) -- (4,3) -- (4,4);
 \draw[loosely dashdotted] (1,1) -- (1,2) -- (1,3) -- (2,3) -- (3,3) -- (4,3) -- (5,3)--(6,3) -- (6,4); 
 \end{tikzpicture}
\end{center}
\ \\[-40pt]
\caption{$L=(l_1,l_2,\ldots,l_r)$}
 \label{fig2}
\end{minipage}
\begin{minipage}{0.5\hsize}
\begin{center}
%{ PICTURE}
% \includegraphics[clip,width=85mm]{Hpattern.pdf}
 \begin{tikzpicture} 
  \node at (0.5,3) {$q$};
     \node at (1,0.5) {$\overline l_2$};
     \node at (3,0.5) {$\overline l_1$};
     \node at (4,0.5) {$p$};
   \node at (1,1) {$\bullet$};
   \node at (1,2) {$\bullet$};
   \node at (1,3) {$\bullet$};
   \node at (1,4) {$\bullet$};    
   \node at (2,1) {$\bullet$};
   \node at (2,2) {$\bullet$};
   \node at (2,3) {$\bullet$};  
   \node at (2,4) {$\bullet$};  
   \node at (3,1) {$\bullet$};
   \node at (3,2) {$\bullet$};
   \node at (3,3) {$\bullet$};  
  \node at (3,4) {$\bullet$};  
   \node at (4,1) {$\bullet$};
   \node at (4,2) {$\bullet$};
   \node at (4,3) {$\bullet$};  
   \node at (4,4) {$\bullet$};  
   \node at (5,1) {$\bullet$};
   \node at (5,2) {$\bullet$};
   \node at (5,3) {$\bullet$};  
   \node at (5,4) {$\bullet$};  
    \node at (6,1) {$\bullet$};
   \node at (6,2) {$\bullet$};
   \node at (6,3) {$\bullet$};  
   \node at (6,4) {$\bullet$};
 \draw (3,1) -- (3,2) -- (4,2) -- (4,3)-- (5,3) --(6,3)--(6,4);
 \draw[loosely dashdotted] (1,1) -- (1,2) -- (1,3) -- (2,3) -- (3,3) -- (4,3)-- (4,4);      
 \end{tikzpicture}
\end{center}
\ \\[-40pt]
\caption{$\overline L=(\overline l_1,\overline l_2,\ldots,l_r)$}
  \label{fig3}
\end{minipage}
\end{tabular}
\end{figure}
}

Then, we obtain that
\[w^N_{\pmb s}(\overline L)=-\prod_{i=1}^{p-1}a_{1i}^{-s_{1i}}\prod_{i=p+1}^{m_2}a_{2i}^{-s_{1(i-1)}}\prod_{i=1}^{p}a_{2i}^{-s_{2i}}\prod_{i=p}^{m_2-1}a_{1i}^{-s_{2(i+1)}}\prod_{i=m_2}^{m_1}a_{1i}^{-s_{1i}}\prod_{j=3}^r\left(\prod_{i=1}^{m_j}a_{ji}^{-s_{ji}}\right).\]
Since one can confirm that \[\sum_{diag}\left[ \prod_{i=1}^{m_1}a_{1i}^{-s_{1i}}\prod_{i=1}^{m_2}a_{2i}^{-s_{2i}}-\prod_{i=1}^{p-1}a_{1i}^{-s_{1i}}\prod_{i=p+1}^{m_2}a_{2i}^{-s_{1(i-1)}}\prod_{i=1}^{p}a_{2i}^{-s_{2i}}\prod_{i=p}^{m_2-1}a_{1i}^{-s_{2(i+1)}}\prod_{i=m_2}^{m_1}a_{1i}^{-s_{1i}}\right]=0,\]
we obtain that $\sum_{diag}X_{\lambda,1}^N(\pmb s)=0$. This shows the assertion.
\end{proof}
By using this lemma, we give some extended Jacobi-Trudi formulas. In the following, we define the main terms and the error terms of our formulas.
To simplify, we introduce the notation $\underline{s}_{\, ia}^{\, ib}$ as an abbreviation for $s_{ia}, s_{i(a+1)},\ldots, s_{ib}$ and $\underline{s}_{\, aj}^{\, bj}$ as an abbreviation for $s_{aj}, s_{(a+1)j},\ldots, s_{bj}$.
Let $\lambda=(\lambda_1,\ldots,\lambda_r)$ be a partition.
Then we define 
\[
\jts(\lambda_1,\ldots,\lambda_r;\underline{s}_{\, 11}^{\,1\lambda_1},\ldots,\underline{s}_{\,r1}^{\,r\lambda_r})={{\rm det}[s^\lambda(i,j)]_{1\le i,j\le r}},\]
where \[{s^\lambda(i,j)}=\left\{\begin{array}{ll}
\zetas(\underline{s}_{\,j1}^{\,j\lambda_j},\underline{s}_{\,(j-1)\lambda_j}^{\,(j-1)\lambda_{j-1}},\ldots,\underline{s}_{\,i\lambda_{i+1}}^{\,i\lambda_i})&\text{ if }i\le j\\
\zetas(\underline{s}_{\,j1}^{\,j(\lambda_i-i+j)})&\text{ if }i>j\text{ and } \lambda_j>i-j\\
1&\text{ if }i>j\text{ and } \lambda_j=i-j\\
0&\text{ if }i>j\text{ and } \lambda_j<i-j.
\end{array}\right.\]
As an example, if $\lambda_i\ge i$ for any $i=1,\ldots,r$, then we have
\begin{align*}
    &\jts(\lambda_1,\ldots,\lambda_r;\underline{s}_{\, 11}^{\,1\lambda_1},\ldots,\underline{s}_{\,r1}^{\,r\lambda_r})\\
    &=\begin{vmatrix}
\zetas(\underline{s}_{\, 11}^{\,1\lambda_1})&\zetas(\underline{s}_{\, 21}^{\,2\lambda_2},\underline{s}_{\,1\lambda_2}^{\,1\lambda_1})&\cdots&\zetas(\underline{s}_{r1}^{\,r\lambda_r},\underline{s}_{\,(r-1)\lambda_r}^{\,(r-1)\lambda_{r-1}},\ldots,\underline{s}_{\,1\lambda_2}^{\,1\lambda_1})\\
\zetas(\underline{s}_{\, 11}^{\,1(\lambda_2-1)})&\zetas(\underline{s}_{\, 21}^{\,2\lambda_2})&\ddots&\zetas(\underline{s}_{\,r1}^{\,r\lambda_r},\underline{s}_{\,(r-1)\lambda_r}^{\,(r-1)\lambda_{r-1}},\ldots,\underline{s}_{\,2\lambda_3}^{\,2\lambda_2})\\
\zetas(\underline{s}_{\, 11}^{\,1(\lambda_3-2)})&\zetas(\underline{s}_{\, 21}^{\,2(\lambda_3-1)})&\ddots&\zetas(\underline{s}_{\,r1}^{\,r\lambda_r},\underline{s}_{\,(r-1)\lambda_r}^{\,(r-1)\lambda_{r-1}},\ldots,\underline{s}_{\,3\lambda_4}^{\,3\lambda_3})\\
\vdots&\ddots&\ddots&\vdots\\
\zetas(\underline{s}_{\, 11}^{\,1(\lambda_r-r+1)})&\zetas(\underline{s}_{\, 21}^{\,2(\lambda_r-r+2)})&\cdots&\zetas(\underline{s}_{\,r1}^{\,r\lambda_r})
\end{vmatrix}.
\end{align*}
Also, we define the error term
\[\jtse(\lambda_1,\ldots,\lambda_r;\underline{s}_{\, 11}^{\,1\lambda_1},\ldots,\underline{s}_{\,r1}^{\,r\lambda_r})=\sum_{\sigma\in S_H^\lambda}\varepsilon_\sigma\prod_{i=1}^r{\zetas}(\theta_i^\sigma(\pmb{s}))-\jts(\lambda_1,\ldots,\lambda_r;\underline{s}_{\, 11}^{\,1\lambda_1},\ldots,\underline{s}_{\,r1}^{\,r\lambda_r}).
\]
Next, we also prepare for the Jacobi-Trudi formula which expresses the Schur multiple zeta function as a determinant in terms of the multiple zeta functions.
Let $\lambda'=(\lambda_1',\ldots,\lambda_s')$ be the conjugate of $\lambda$.
Then we define $\jt(\lambda_1',\ldots,\lambda_s';\underline{s}_{\, 11}^{\,\lambda_1'1},\ldots,\underline{s}_{\,1s}^{\,\lambda_s's})$ by 
\[
\jt(\lambda_1',\ldots,\lambda_s';\underline{s}_{\, 11}^{\,\lambda_1'1},\ldots,\underline{s}_{\,1s}^{\,\lambda_s's})={{\rm det}[s^{\lambda'}(i,j)]_{1\le i,j\le s},}\]
where \[{s^{\lambda'}(i,j)}=\left\{\begin{array}{ll}
\zeta(\underline{s}_{\,1j}^{\,\lambda_j'j},\underline{s}_{\,\lambda_j'(j-1)}^{\,\lambda_{j-1}'(j-1)},\ldots,\underline{s}_{\,\lambda_{i+1}'i}^{\,\lambda_i'i})&\text{ if }i\le j\\
\zeta(\underline{s}_{\,1j}^{\,(\lambda_i'-i+j)j})&\text{ if }i>j\text{ and } \lambda_i'>i-j\\
1&\text{ if }i>j\text{ and } \lambda_i'=i-j\\
0&\text{ if }i>j\text{ and } \lambda_i'<i-j.
\end{array}\right.\]
As an example, if $\lambda_i'\ge i$ for any $i=1,\ldots,s$, then we have
\[
\jt(\lambda_1',\ldots,\lambda_s';\underline{s}_{\, 11}^{\lambda_1'1},\ldots,\underline{s}_{\,1s}^{\,\lambda_s's})=\begin{vmatrix}
\zeta(\underline{s}_{\, 11}^{\,\lambda_1'1})&\zetas(\underline{s}_{\,12}^{\,\lambda_2'2},\underline{s}_{\,\lambda_2'1}^{\,\lambda_1'1})&\cdots&\zeta(\underline{s}_{\,1s}^{\,\lambda_s's},\underline{s}_{\,\lambda_s'(s-1)}^{\,\lambda_{s-1}'(s-1)},\ldots,\underline{s}_{\,\lambda_2'1}^{\,\lambda_1'1})\\
\zeta(\underline{s}_{\, 11}^{\,(\lambda_2'-1)1})&\zetas(\underline{s}_{\,12}^{\,\lambda_2'2})&\ddots&\zeta(\underline{s}_{\,1s}^{\,\lambda_s's},\underline{s}_{\,\lambda_s'(s-1)}^{\,\lambda_{s-1}'(s-1)},\ldots,\underline{s}_{\,\lambda_3'2}^{\,\lambda_2'2})\\
\zeta(\underline{s}_{\, 11}^{\,(\lambda_3'-2)}1)&\zetas(\underline{s}_{\,12}^{\,(\lambda_3'-1)2})&\ddots&\zeta(\underline{s}_{\,1s}^{\,\lambda_s's},\underline{s}_{\,\lambda_s'(s-1)}^{\,\lambda_{s-1}'(s-1)},\ldots,\underline{s}_{\,\lambda_4'3}^{\,\lambda_3'3})\\
\vdots&\ddots&\ddots&\vdots\\
\zeta(\underline{s}_{\, 11}^{\,(\lambda_s'-s+1)})&\zeta(\underline{s}_{\,12}^{\,(\lambda_s'-s+2)2})&\cdots&\zeta(\underline{s}_{\,1s}^{\,\lambda_s's})
\end{vmatrix}.
\]
As in the definition of $\jtse$, we define 
\[\jte(\lambda_1',\ldots,\lambda_s';\underline{s}_{\, 11}^{\lambda_1'1},\ldots,\underline{s}_{\,1s}^{\,\lambda_s's})=\sum_{\sigma\in S_E^\lambda}\varepsilon_\sigma\prod_{i=1}^r{\zetas}(\theta_i^\sigma(\pmb{s}))-\jts(\lambda_1',\ldots,\lambda_s';\underline{s}_{\, 11}^{\lambda_1'1},\ldots,\underline{s}_{\,1s}^{\,\lambda_s's}).
\]

\begin{exam}[$\lambda=(4,3,1)$]
Let $\lambda=(4,3,1)$ and $\lambda'=(3,2,2,1)$. Then 
\begin{align*}
\jts(4,3,1;\underline{s}_{\, 11}^{\,14},\underline{s}_{\, 21}^{\,23},s_{31})&=\begin{vmatrix}
\zetas(\underline{s}_{\, 11}^{\,14})&\zetas(\underline{s}_{\, 21}^{\,23},\underline{s}_{\,13}^{\,14})&\zetas(s_{31},\underline s_{\,21}^{\,23},\underline s_{\,13}^{\,14})\\
\zetas(\underline{s}_{\, 11}^{\,12})&\zetas(\underline{s}_{\, 21}^{\,23})&\zetas(s_{31},\underline{s}_{\, 21}^{\,23})\\
0&1&\zetas(s_{31})
\end{vmatrix},
\intertext{and $\jtse(4,3,1;\underline{s}_{\, 11}^{\,14},\underline{s}_{\, 21}^{\,23},s_{31})=0$, and}
\jt(3,2,2,1;\underline{s}_{\, 11}^{\,31},\underline{s}_{\,12}^{\,22},\underline{s}_{\,13}^{\,23},s_{14})&=\begin{vmatrix}
\zeta(\underline{s}_{\, 11}^{\,31})&\zeta(\underline{s}_{\,12}^{\,22},\underline{s}_{\, 21}^{\,31})&\zeta(\underline{s}_{\,13}^{\,23},s_{22},\underline{s}_{\, 21}^{\,31})&\zeta(s_{14},\underline{s}_{\,13}^{\,23},s_{22},\underline{s}_{\, 21}^{\,31})\\
\zeta({s}_{11})&\zeta(\underline{s}_{\,12}^{\,22})&\zeta(\underline{s}_{\,13}^{\,23},s_{22})&\zeta(s_{14},\underline{s}_{\,13}^{\,23},s_{22})\\
1&\zeta({s}_{12})&\zeta(\underline{s}_{\,13}^{\,23})&\zeta(s_{14},\underline{s}_{\,13}^{\,23})\\
0&0&1&\zeta(s_{14})\\
\end{vmatrix},
\intertext{and}
    \jte(3,2,2,1;\underline{s}_{\, 11}^{\,31},\underline{s}_{\,12}^{\,22},\underline{s}_{\,13}^{\,23},s_{14})&=
   \begin{vmatrix}
\zeta(s_{12},\underline{s}_{\, 11}^{\,31})-\zeta(\underline{s}_{\,12}^{\,22},\underline{s}_{\, 21}^{\,31})&\zeta(\underline{s}_{\,13}^{\,23},s_{22},\underline{s}_{\, 21}^{\,31})&\zeta(s_{14},\underline{s}_{\,13}^{\,23},s_{22},\underline{s}_{\, 21}^{\,31})\\
\zeta(s_{12},s_{11})-\zeta(\underline{s}_{\,12}^{\,22})&\zeta(\underline{s}_{\,13}^{\,23},s_{22})&\zeta(s_{14},\underline{s}_{\,13}^{\,23},s_{22})\\
0&1&\zeta(s_{14})\\
\end{vmatrix}.
\end{align*}
\end{exam}

In the following three subsections, we show the extended Jacobi-Trudi formulas for shape $(m,n)$, $(m,n,\{1\}^X)$ and $(m,n,\{2\}^X)$ respectively. 
\subsection{In case of shape $(m,n)$}
First, we give the Jacobi-Trudi formulas for the Schur multiple zeta functions of shape $(m,n)$. This case is fundamental. 
\begin{theorem}
\label{mn}
For $\lambda=(m,n)$ and $(s_{ij})\in W_{\lambda,H}$, it holds that
\[
  \sum_{diag}\begin{ytableau}
  s_{11}&s_{12}&\cdots&\cdots&\cdots&s_{1m}\\
  s_{21}&s_{22}&\cdots&s_{2n}\\
\end{ytableau}=\sum_{diag}\jts(m,n;\underline{s}_{\, 11}^{\,1m},\underline{s}_{\, 21}^{\,2n})=\sum_{diag}\begin{vmatrix}
\zetas(\underline{s}_{\, 11}^{\,1m})&\zetas(\underline{s}_{\, 21}^{\,2n},\underline{s}_{1n}^{1m})\\
\zetas(\underline{s}_{\, 11}^{\,1(n-1)})&\zetas(\underline{s}_{\, 21}^{\,2n})
\end{vmatrix}.
\]
\end{theorem}
\begin{proof}
By Lemma \ref{diag}, we find that 
\begin{align*}
    \sum_{diag}\zeta_{\lambda}(\pmb s)&=\sum_{diag}{\zetas}(\theta_1^{\rm id}(\pmb s)){\zetas}(\theta_2^{\rm id}(\pmb s))-{\zetas}(\theta^{(12)}_1(\pmb s)){\zetas}(\theta^{(12)}_2(\pmb s))\\
&=\sum_{diag}\begin{vmatrix}
\zetas(\underline{s}_{\, 11}^{\,1m})&\zetas(\underline{s}_{\, 21}^{\,2n},\underline{s}_{\,1n}^{\,1m})\\
\zetas(\underline{s}_{\, 11}^{\,1(n-1)})&\zetas(\underline{s}_{\, 21}^{\,2n})
\end{vmatrix}.
\end{align*}
%Now we focus on the error $X_{\lambda,1}^N(\pmb s)$. For $L = (l_1, l_2) \in H^N_{\lambda}\setminus H^N_{\lambda,0}$ of type ${\rm id}$, we consider the rightmost intersection point $(p,q)$ appearing in $L$.
%Then we expand $w^N_{\pmb s}(L)$ as \[\prod_{i=1}^ma_{1i}^{-x_{1i}}\prod_{i=1}^na_{2i}^{-x_{2i}}.\]
%On the other hand, we consider the pair of path $\overline{L}=(\overline l_1,\overline{l_2})$.
%Here, $\overline l_i$ follows $l_i$ until it meets the first intersection point $(p,q)$ and after that follows the other pass $l_j$ to the end. 
%Then, we obtain that
%\[w^N_{\pmb s}(\overline L)=-\prod_{i=1}^{p-1}a_{1i}^{-x_{1i}}\prod_{i=p+1}^{n}a_{2i}^{-x_{1i-1}}\prod_{i=1}^{p}a_{2i}^{-x_{2i}}\prod_{i=p}^{n-1}a_{1i}^{-x_{2i+1}}\prod_{i=n}^{m}a_{1i}^{-x_{1i}}.\]
%Since one can confirm that \[\sum_{diag}\left[ \prod_{i=1}^ma_{1i}^{-x_{1i}}\prod_{i=1}^na_{2i}^{-x_{2i}}-\prod_{i=1}^{p-1}a_{1i}^{-x_{1i}}\prod_{i=p+1}^{n}a_{2i}^{-x_{1i-1}}\prod_{i=1}^{p}a_{2i}^{-x_{2i}}\prod_{i=p}^{n-1}a_{1i}^{-x_{2i+1}}\prod_{i=n}^{m}a_{1i}^{-x_{1i}}\right]=0,\]
%we obtain that $\sum_{diag}X_{\lambda,1}^N(\pmb s)=0$. 
We obtain the identity of this theorem.
\end{proof}

\subsection{In case of shape $(m,n,\{1\}^X)$}
As one of the generalizations of the previous subsection, we consider the extended Jacobi-Trudi formula for the Schur multiple zeta functions of shape $(m,n,\{1\}^X)$. This is also the case with no error terms.
\begin{theorem}
\label{mnx}
For $X\ge2$, $\lambda=(m,n,\{1\}^{X-2})$ and $(s_{ij})\in W_{\lambda,H}$, it holds that
\begin{align*}
  \sum_{diag}\begin{ytableau}
  s_{11}&s_{12}&\cdots&\cdots&\cdots&s_{1m}\\
  s_{21}&s_{22}&\cdots&s_{2n}\\
s_{31}\\
\vdots\\
s_{X1}
\end{ytableau}&=\sum_{diag}\jts(m,n,\{1\}^{X-2};\underline{s}_{\, 11}^{\,1m},\underline{s}_{\, 21}^{\,2n},s_{31},\ldots,s_{X1}).
\end{align*}
\end{theorem}
\begin{proof}
For the case of $X=2$, it is shown by Theorem \ref{mn}.
Now, we assume the assertion holds for $X=r-1$ and will show it holds in the case $X=r$.
Let $L=(l_1,\ldots,l_r)$ be a $r$-tuple of paths. By the definition of paths, the path $l_r$ starts from $A_r=(r,1)$ and ends to $B_r=(r+m,N)$ or $B_{r-1}=(r-1+n,N)$. We denote by $G_i$ the set of all permutation $\sigma$ with $\sigma(r)=i$.
Then, we can split the sum of the right-hand side of Lemma \ref{diag} as
\begin{equation}
    \label{1stdecom}
     \sum_{\sigma\in S_H^\lambda}\varepsilon_\sigma\prod_{i=1}^r{\zetas}^N(\theta_i^\sigma(\pmb{s}))=\sum_{\sigma\in G_r}\varepsilon_{\sigma}\prod_{i=1}^{r}{\zetas}^N(\theta^{\sigma}_i(\pmb s))+\sum_{\sigma\in G_{r-1}}\varepsilon_{\sigma}\prod_{i=1}^{r}{\zetas}^N(\theta^{\sigma}_i(\pmb s)).
\end{equation}
\begin{align*}
    \intertext{Since we can compute ${\zetas}^N(\theta^{\sigma}_r(\pmb s))$ easily, we get}
    \sum_{\sigma\in S_H^\lambda}\varepsilon_\sigma\prod_{i=1}^r{\zetas}^N(\theta_i^\sigma(\pmb{s}))&={\zetas}^N(\underline{s}_{\, 11}^{1m})\sum_{\sigma\in G_r}\varepsilon_{\sigma}\prod_{i=1}^{r-1}{\zetas}^N(\theta^{\sigma}_i(\pmb s))+{\zetas}^N(\underline{s}_{\, 11}^{1 (n-1)})\sum_{\sigma\in G_{r-1}}\varepsilon_{\sigma}\prod_{i=1}^{r-1}{\zetas}^N(\theta^{\sigma}_i(\pmb s)).
    \end{align*}
    By the induction hypothesis, we obtain that
    \begin{align*}
        \sum_{diag}\sum_{\sigma\in S_H^\lambda}\varepsilon_\sigma\prod_{i=1}^r{\zetas}^N(\theta_i^\sigma(\pmb{s}))=&\sum_{diag}\left[{\zetas}^N(\underline{s}_{\, 11}^{\,1m})\jts(n,\{1\}^{X-2};\underline{s}_{\, 21}^{\,2n},s_{31},\ldots,s_{X1})\right.\\
    &-\left.{\zetas}^N(\underline{s}_{\, 11}^{\,1(n-1)})\jts(m+1,\{1\}^{X-2};\underline{s}_{\, 21}^{\,2n},\underline{s}_{\,1n}^{\,1m},s_{31},\ldots,s_{X1})\right]\\
   =&\sum_{diag}\jts(m,n,\{1\}^{X-2};\underline{s}_{\, 11}^{\,1m},\underline{s}_{\, 21}^{\,2n},s_{31},\ldots,s_{X1}).
    \end{align*}
Applying Lemma \ref{diag} to the left-hand side, we obtain 
the desired conclusion.
\end{proof}
\subsection{In case of shape $(m,n,\{2\}^X)$}
We face some error terms of the extended Jacobi-Trudi formula for the Schur multiple zeta functions of shape $(m,n,\{2\}^X)$. We can get a recursive formula to characterize these error terms.
\begin{theorem}
\label{mnx2}
For $X\ge2$, $\lambda=(m,n,\{2\}^{X-2})$ and $(s_{ij})\in W_{\lambda,H}$, it holds that
\begin{align*}
  &\sum_{diag}\begin{ytableau}
  s_{11}&s_{12}&\cdots&\cdots&\cdots&s_{1m}\\
  s_{21}&s_{22}&\cdots&s_{2n}\\
s_{31}&s_{32}\\
  \vdots&\vdots\\
  s_{X1}&s_{X2}
\end{ytableau}\\
&=\sum_{diag}[\jts(m,n,\{2\}^{X-2};\underline{s}_{\,11}^{\,1m},\underline{s}_{\, 21}^{\,2n},\underline{s}_{\, 31}^{\,32},\ldots,\underline{s}_{\,X1}^{\,X2})+\jtse(m,n,\{2\}^{X-2};\underline{s}_{\, 11}^{\,1m},\underline{s}_{\, 21}^{\,2n},\underline{s}_{\, 31}^{\,32},\ldots,\underline{s}_{\,X1}^{\,X2})],
\end{align*}
{ where $\jtse(m,n;\underline{s}_{\,11}^{\,1m},\underline{s}_{\, 21}^{\,2n})=0$ and for $X\ge3$, the error term $\jtse$ satisfies the recursive formula}
\begin{align*}
&\jtse(m,n,\{2\}^{X-2};\underline{s}_{\,11}^{\,1m},\underline{s}_{\, 21}^{\,2n},\underline{s}_{\, 31}^{\,32},\ldots,\underline{s}_{\,X1}^{\,X2})\\
&=\sum_{diag}[\zetas(\underline{s}_{\, 11}^{\,1m})\jtse(n,\{2\}^{X-2};\underline{s}_{\, 21}^{\,2n},\underline{s}_{\, 31}^{\,32},\ldots,\underline{s}_{\,X1}^{\,X2})\\
         &-\zetas(\underline{s}_{\, 11}^{\,1(n-1)})\jtse(m+1,\{2\}^{X-2};\underline{s}_{\, 21}^{\,2n},\underline{s}_{\,1n}^{\,1m},\underline{s}_{\, 31}^{\,32},\ldots,\underline{s}_{\,X1}^{\,X2})\\
         &-\jts(m+1,n+1,\{2\}^{X-3};\underline s_{21}^{2n},\underline{s}_{\,1n}^{\,1m},\underline{s}_{\, 31}^{\,32},\underline{s}_{\,22}^{\,2n},\underline{s}_{\, 41}^{\,42},\ldots,\underline{s}_{\,X1}^{\,X2})\\
        &+\jts(m+1,n+1,\{2\}^{X-3};s_{21},\underline{s}_{\,11}^{\,1m},\underline{s}_{\, 31}^{\,32},\underline{s}_{\,22}^{\,2n},\underline{s}_{\, 41}^{\,42},\ldots,\underline{s}_{\,X1}^{\,X2})\\
         &+\jtse(m+1,n+1,\{2\}^{X-3};s_{21},\underline{s}_{\,11}^{\,1m},\underline{s}_{\, 31}^{\,32},\underline{s}_{\,22}^{\,2n},\underline{s}_{\, 41}^{\,42},\ldots,\underline{s}_{\,X1}^{\,X2})].
\end{align*}
\end{theorem}
\begin{proof}
The similar discussion in the proof of Theorem \ref{mn} holds.
The difference is $G_{r-2}$ term in identity (\ref{1stdecom}):
\begin{align*}
    &\sum_{\sigma\in S_H^\lambda}\varepsilon_\sigma\prod_{i=1}^r{\zetas}^N(\theta_i^\sigma(\pmb{s}))\\
    &=\sum_{\sigma\in G_r}\varepsilon_{\sigma}\prod_{i=1}^{r}{\zetas}^N(\theta^{\sigma}_i(\pmb s))+\sum_{\sigma\in G_{r-1}}\varepsilon_{\sigma}\prod_{i=1}^{r}{\zetas}^N(\theta^{\sigma}_i(\pmb s))+\sum_{\sigma\in G_{r-2}}\varepsilon_{\sigma}\prod_{i=1}^{r}{\zetas}^N(\theta^{\sigma}_i(\pmb s)).
    \end{align*}
Simple calculation leads to
    \begin{align*}
    &\sum_{\sigma\in S_H^\lambda}\varepsilon_\sigma\prod_{i=1}^r{\zetas}^N(\theta_i^\sigma(\pmb{s}))\\
    &={\zetas}^N(\underline{s}_{\, 11}^{\,1m})\sum_{\sigma\in G_r}\varepsilon_{\sigma}\prod_{i=1}^{r-1}{\zetas}^N(\theta^{\sigma}_i(\pmb s))+{\zetas}^N(\underline{s}_{\, 11}^{\,1(n-1)})\sum_{\sigma\in G_{r-1}}\varepsilon_{\sigma}\prod_{i=1}^{r-1}{\zetas}^N(\theta^{\sigma}_i(\pmb s))\\
    &+\sum_{\sigma\in G_{r-2}}\varepsilon_{\sigma}\prod_{i=1}^{r-1}{\zetas}^N(\theta^{\sigma}_i(\pmb s))\\
    &=:{\zetas}^N(\underline{s}_{\, 11}^{\,1m})S_r+{\zetas}^N(\underline{s}_{\, 11}^{\,1(n-1)})S_{r-1}+S_{r-2}.
\end{align*}
First we consider $S_r$. By the induction hypothesis and Lemma \ref{diag}, we find
\begin{align*}
    \sum_{diag}S_r&=\sum_{diag}\sum_{\sigma\in G_r}\varepsilon_{\sigma}\prod_{i=1}^{r-1}{\zetas}^N(\theta^{\sigma}_i(\pmb s))\\
    &=\sum_{diag}[\jts(n,\{2\}^{X-2};\underline{s}_{\, 21}^{\,2n},\underline{s}_{\, 31}^{\,32},\ldots,\underline{s}_{\,X1}^{\,X2})+\jtse(n,\{2\}^{X-2};\underline{s}_{\, 21}^{\,2n},\underline{s}_{\, 31}^{\,32},\ldots,\underline{s}_{\,X1}^{\,X2})].
\end{align*}
Similarly, we get 
\begin{align*}
    \sum_{diag}S_{r-1}=&-\sum_{diag}[\jts(m+1,\{2\}^{X-2};\underline{s}_{\, 21}^{\,2n},\underline{s}_{1n}^{\,1m},\underline{s}_{\, 31}^{\,32},\ldots,\underline{s}_{\,X1}^{\,X2})\\
    &+\jtse(m+1,\{2\}^{X-2};\underline{s}_{\, 21}^{\,2n},\underline{s}_{\,1n}^{\,1m},\underline{s}_{\, 31}^{\,32},\ldots,\underline{s}_{\,X1}^{\,X2})],
\end{align*}
and
\begin{align*}
\sum_{diag}S_{r-2}=&\sum_{diag}[\jts(m+1,n+1,\{2\}^{X-3};s_{21},\underline{s}_{\,11}^{\,1m},\underline{s}_{\, 31}^{\,32},\underline{s}_{\,22}^{\,2n},\underline{s}_{\, 41}^{\,42},\ldots,\underline{s}_{\,X1}^{\,X2})\\
         &+\jtse(m+1,n+1,\{2\}^{X-3};s_{21},\underline{s}_{\,11}^{\,1m},\underline{s}_{\, 31}^{\,32},\underline{s}_{\,22}^{\,2n},\underline{s}_{\, 41}^{\,42},\ldots,\underline{s}_{\,X1}^{\,X2})]\\
         =&\sum_{diag}[\jts(m+1,n+1,\{2\}^{X-3};\underline s_{21}^{2n},\underline{s}_{\,1n}^{\,1m},\underline{s}_{\, 31}^{\,32},\underline{s}_{\,22}^{\,2n},\underline{s}_{\, 41}^{\,42},\ldots,\underline{s}_{\,X1}^{\,X2})\\
         &-\jts(m+1,n+1,\{2\}^{X-3};\underline s_{21}^{2n},\underline{s}_{\,1n}^{\,1m},\underline{s}_{\, 31}^{\,32},\underline{s}_{\,22}^{\,2n},\underline{s}_{\, 41}^{\,42},\ldots,\underline{s}_{\,X1}^{\,X2})\\
        &+\jts(m+1,n+1,\{2\}^{X-3};s_{21},\underline{s}_{\,11}^{\,1m},\underline{s}_{\, 31}^{\,32},\underline{s}_{\,22}^{\,2n},\underline{s}_{\, 41}^{\,42},\ldots,\underline{s}_{\,X1}^{\,X2})\\
         &+\jtse(m+1,n+1,\{2\}^{X-3};s_{21},\underline{s}_{\,11}^{\,1m},\underline{s}_{\, 31}^{\,32},\underline{s}_{\,22}^{\,2n},\underline{s}_{\, 41}^{\,42},\ldots,\underline{s}_{\,X1}^{\,X2})].
\end{align*}
Therefore, we find that the sum of the first $\jts$ terms of the each three identities is the same with
\[\jts(m,n,\{2\}^{X-2};\underline{s}_{\,11}^{\,1m},\underline{s}_{\,21}^{\,2n},\underline{s}_{\,31}^{\,32},\ldots,\underline{s}_{\,X1}^{\,X2}).\]
On the other hand, the sum of the other terms of the above three identities is
\begin{align*}
    &\sum_{diag}[
         {\zetas}^N(\underline{s}_{\, 11}^{\,1m})\jtse(n,\{2\}^{X-2};\underline{s}_{\, 21}^{\,2n},\underline{s}_{\, 31}^{\,32},\ldots,\underline{s}_{\,X1}^{\,X2})\\
         &-{\zetas}^N(\underline{s}_{\, 11}^{\,1(n-1)})\jtse(m+1,\{2\}^{X-2};\underline{s}_{\, 21}^{\,2n},\underline{s}_{\,1n}^{\,1m},\underline{s}_{\, 31}^{\,32},\ldots,\underline{s}_{\,X1}^{\,X2})\\ 
         &-\jts(m+1,n+1,\{2\}^{X-3};\underline s_{21}^{2n},\underline{s}_{\,1n}^{\,1m},\underline{s}_{\, 31}^{\,32},\underline{s}_{\,22}^{\,2n},\underline{s}_{\, 41}^{\,42},\ldots,\underline{s}_{\,X1}^{\,X2})\\
        &+\jts(m+1,n+1,\{2\}^{X-3};s_{21},\underline{s}_{\,11}^{\,1m},\underline{s}_{\, 31}^{\,32},\underline{s}_{\,22}^{\,2n},\underline{s}_{\, 41}^{\,42},\ldots,\underline{s}_{\,X1}^{\,X2})\\
         &+\jtse(m+1,n+1,\{2\}^{X-3};s_{21},\underline{s}_{\,11}^{\,1m},\underline{s}_{\, 31}^{\,32},\underline{s}_{\,22}^{\,2n},\underline{s}_{\, 41}^{\,42},\ldots,\underline{s}_{\,X1}^{\,X2})].
\end{align*}
Thus, we obtain the desired recursive formula.
This completes the proof of this theorem.
\end{proof}
We give an example for shape $\lambda=(3,2,2)$.
\begin{exam}[$\lambda=(3,2,2)$]For $(s_{ij})\in W_{\lambda,H}$, it holds that
\begin{align*}
    \sum_{diag}\begin{ytableau}
  s_{11}&s_{12}&s_{13}\\
  s_{21}&s_{22}\\
  s_{31}&s_{32}
\end{ytableau}=\sum_{diag}&\left[\begin{vmatrix}
\zetas(s_{11},s_{12},s_{13})&\zetas(s_{21},s_{22},s_{12},s_{13})&\zetas(s_{31},s_{32},s_{22},s_{12},s_{13})\\
\zetas(s_{11})&\zetas(s_{21},s_{22})&\zetas(s_{31},s_{32},s_{22})\\
1&\zetas(s_{21})&\zetas(s_{31},s_{32})
\end{vmatrix}\right.\\
&+(\zetas(s_{21},s_{11},s_{12},s_{13})-\zetas(s_{21},s_{22},s_{12},s_{13}))\zetas(s_{31},s_{32},s_{22})\\
&\left.-(\zetas(s_{21},s_{11})-\zetas(s_{21},s_{22}))\zetas(s_{31},s_{32},s_{22},s_{12},s_{13})\right].
\end{align*}
\end{exam}
Considering the transpose of the statement of Theorem \ref{mnx} and Theorem \ref{mnx2}, we obtain the extended Jacobi-Trudi formula which expresses the Schur multiple zeta function as a determinant in terms of multiple zeta functions.
{Let $\lambda$ be a partition and $\lambda'=(\lambda_1'\ldots,\lambda_s')$ be the conjugate of $\lambda$. Then we define a set $W_{\lambda,E}$ by
\[
 W_{\lambda,E}
=
\left\{{\pmb s}=(s_{ij})\in T(\lambda,\mathbb{C})\,\left|\,
\begin{array}{l}
 \text{$\Re(s_{ij})\ge 1$ for all $(i,j)\in D(\lambda) \setminus E(\lambda)$} \\[3pt]
 \text{$\Re(s_{ij})>1$ for all $(i,j)\in E(\lambda)$}
\end{array}
\right.
\right\},
\]
where $E(\lambda)=\{(i,j)\in D(\lambda)~|~i-j\in\{i-\lambda_i'~|~1\le i\le s\} \}$. In the following two Corollary, we assume corresponding $\pmb s\in W_{\lambda,E}$ in place of $\pmb s\in W_{\lambda,H}$.} 
\begin{cor}
\label{mxy}
For $m\ge n\ge1$, $\lambda=(X,\{2\}^{n-1},\{1\}^{m-n})$ and $(s_{ij})\in W_{\lambda,E}$, it holds that
\[\sum_{diag}\begin{ytableau}
  s_{11}&s_{12}&\cdots&s_{1X}\\
s_{21}&s_{22}\\
\vdots&\vdots\\
\vdots&s_{n2}\\
\vdots\\
s_{m1}
\end{ytableau}=\sum_{diag}\jt(m,n,\{2\}^{X-2};\underline{s}_{\, 11}^{\,m1},\underline{s}_{\,12}^{\,n2},s_{13},\ldots,s_{1X}).\]
\end{cor}
\begin{cor}
\label{mxy2}
For $m\ge n\ge2$, $\lambda=(\{X\}^2,\{2\}^{n-2},\{1\}^{m-n})$ and $(s_{ij})\in W_{\lambda,E}$,
\begin{align*}
 &\sum_{diag}\begin{ytableau}
  s_{11}&s_{12}&\cdots&s_{1X}\\
  s_{21}&s_{22}&\cdots&s_{2X}\\
  \vdots&\vdots\\
   \vdots&s_{n2}\\
    \vdots\\
  s_{m1}
\end{ytableau}\\
&=\sum_{diag}[\jt(m,n,\{2\}^{X-2};\underline{s}_{\, 11}^{\,m1},\underline{s}_{\,12}^{\,n2},\underline{s}_{\,13}^{\,23},\ldots,\underline{s}_{\,1X}^{\,2X})+\jte(m,n,\{2\}^{X-2};\underline{s}_{\, 11}^{\,m1},\underline{s}_{\,12}^{\,n2},\underline{s}_{\,13}^{\,23},\ldots,\underline{s}_{\,1X}^{\,2X})],
\end{align*}
{ where $\jte(m,n;\underline{s}_{\,11}^{\,m1},\underline{s}_{\, 21}^{\,n2})=0$ and for $X\ge3$, 
$\jte$ satisfies the recursive formula}
\begin{align*}
&\jte(m,n,\{2\}^{X-2};\underline{s}_{\, 11}^{\,m1},\underline{s}_{\,12}^{\,n2},\underline{s}_{\,13}^{\,23},\ldots,\underline{s}_{\,1X}^{\,2X})\\
&=\sum_{diag}[\zeta(\underline{s}_{\, 11}^{\,m1})\jte(n,\{2\}^{X-2};\underline{s}_{\,12}^{\,n2},\underline{s}_{\,13}^{\,23},\ldots,\underline{s}_{\,1X}^{\,2X})\\
&-\zeta(\underline{s}_{\, 11}^{\,(n-1)1})\jte(m+1,\{2\}^{X-2};\underline{s}_{\,12}^{\,n2},\underline{s}_{\,n1}^{\,m1},\underline{s}_{\,13}^{\,23},\ldots,\underline{s}_{\,1X}^{\,2X})\\ 
&-\jt(m+1,n+1,\{2\}^{X-3};\underline s_{12}^{n2},\underline{s}_{\,n1}^{\,m1},\underline{s}_{\, 13}^{\,23},\underline{s}_{\,22}^{\,n2},\underline{s}_{\, 14}^{\,24},\ldots,\underline{s}_{\,1X}^{\,2X})\\
&+\jts(m+1,n+1,\{2\}^{X-3};s_{12},\underline{s}_{\,11}^{\,m1},\underline{s}_{\, 13}^{\,23},\underline{s}_{\,22}^{\,n2},\underline{s}_{\, 14}^{\,24},\ldots,\underline{s}_{\,1X}^{\,2X})\\
&+\jtse(m+1,n+1,\{2\}^{X-3};s_{12},\underline{s}_{\,11}^{\,m1},\underline{s}_{\, 13}^{\,23},\underline{s}_{\,22}^{\,n2},\underline{s}_{\, 14}^{\,24},\ldots,\underline{s}_{\,1X}^{\,2X})].
\end{align*}
\end{cor}
\subsection{Identities involving multiple zeta (star) functions}
%Applying Corollary \ref{mxy} and \ref{mxy2}, we have the following identity involving multiple zeta functions and multiple zeta star functions.
Combining Theorem \ref{mnx} and Corollary \ref{mxy}, we have the following identity involving multiple zeta and multiple zeta star functions. 
\begin{theorem}
For $m\ge 2$, $X\ge2$, $\lambda=(m,n.\{1\}^{X-2})$ and $(s_{ij})\in W_{\lambda,E}\cap W_{\lambda,H}$, it holds that
\begin{align*}
    &\sum_{diag}\jts(m,2,\{1\}^{X-2};s_{11},\ldots,s_{1m},s_{21},s_{22},s_{31},\ldots,s_{X1})\\
    &=\sum_{diag}\jt(X,2,\{1\}^{m-2};s_{11},\ldots,s_{X1},s_{12},s_{22},s_{13},\ldots,s_{1m}).
\end{align*}
\end{theorem}
\begin{exam}For $\Re(a),\Re(b),\Re(c),\Re(d)>1$, it holds that
\[\sum_{diag}\begin{vmatrix}
\zetas(a,b)&\zetas(c,d,b)\\
\zetas(a)&\zetas(c,d)
\end{vmatrix}=\sum_{diag}\begin{ytableau}
  a&b\\
  c&d
\end{ytableau}=\sum_{diag}\begin{vmatrix}
\zeta(a,c)&\zeta(b,d,c)\\
\zeta(a)&\zeta(b,d)
\end{vmatrix}.\]
In other words, 
\begin{align*}
    &\zetas(a,b)\zetas(c,d)-\zetas(a)\zetas(c,d,b)+\zetas(d,b)\zetas(c,a)-\zetas(d)\zetas(c,a,b)\\
    &=\zeta(a,c)\zeta(b,d)-\zeta(a)\zeta(b,d,c)+\zeta(d,c)\zeta(b,a)-\zeta(d)\zeta(b,a,c).
\end{align*}
\end{exam}

\section{Pieri formula for hook type}\label{pieri}
In this section, we show the Pieri formulas for hook type by using the extended Jacobi-Trudi formulas shown in the previous section. 
\begin{theorem}[$(\lambda=(\ell,\{1\}^{k}))\cdot (\pmb{h}_{m}=(m))$]
For a positive integer $\ell$ and non-negative integers $k$ and $m$,
%{it holds that  if $\Re(x_k)>1, \Re (y_1)>1,\ldots,\Re( y_\ell)>1, \Re (z_1)>1,\ldots,\Re (z_{\ell-1})>1$ and $\Re(z_m)>1$, then}
%{if $\Re(y_j)>1\ (1\le j\le \ell), \Re(z_p)>1\ (1\le p\le \ell-1)$,}
{we assume $\Re(x_k),\Re(y_i), \Re(z_j)>1\ (1\le i\le \ell,1\le j\le \ell-1, j=m)$ and the real parts of other variables are greater than or equal to $1$. Then it holds that}
\[
\ytableausetup{boxsize=1.5em}
\sum_{sym}\begin{ytableau}
  y_1&\cdots&y_{\ell}\\
  x_1\\
  \vdots\\
  x_k
\end{ytableau}\cdot\begin{ytableau}
  z_1&z_2&\cdots&z_m
\end{ytableau}=\sum_{sym}\sum_{\pmb u_\mu\in \pmb U_H}\zeta_{\mu}(\pmb u_\mu),
\] 
where $\displaystyle{\sum_{sym}}$ means the summation over the permutation of $S=\{y_{1},\ldots,y_{{\ell}},z_1,\ldots,z_{{\ell}-1}\}$ as indeterminates.
\end{theorem}
\begin{proof}
It suffices to show the case $k=1$ and ${\ell}=m$, since for any $i\ge2$ and $j> {\ell}$, $x_i$ and $z_j$ do not have an effect on all sums.  By the Jacobi-Trudi formula for SMZFs, the left-hand side of the assertion is 
 \[
    \sum_{W=S}\zetas(w_1,\ldots,w_{\ell})\zetas(x_1)\zetas(w_{{\ell}+1},\ldots,w_{2{\ell}-1},z_{\ell})-\zetas(x_1,w_1,\ldots,w_{\ell})\zetas(w_{\ell+1},\ldots,w_{2{\ell}-1},z_{\ell}),
 \]
 where $W=\{w_1,\ldots,w_{2\ell-1}\}$.
On the other hand, using Theorem \ref{mnx}, we find that the right-hand side of the assertion is 
 \begin{align*}
    &\sum_{W=S}\sum_{i={\ell}-1}^{2{\ell}-2}\left[\begin{vmatrix}
     \zetas(\underline{w}_{\, 1}^{\, i+1},z_{\ell})&\zetas(x_1,\underline{w}_{\, i}^{\, 2{\ell}-1},\underline{w}_{\, 2{\ell}-1-i}^{\,i},z_{\ell})\\
      \zetas(\underline{w}_{\, 1}^{\,2{\ell}-2-i})&\zetas(x_1,\underline{w}_{\,i+2}^{\,2{\ell}-1})
      \end{vmatrix}\right.\\
    &+\left.\begin{vmatrix}
     \zetas(\underline{w}_{\,1}^{\,i},z_{\ell})&\zetas(\underline{w}_{\,i+1}^{\,2{\ell}-1},\underline{w}_{\,2{\ell}-1-i}^{\,i},z_{\ell})&\zetas(x_1,\underline{w}_{\,i+1}^{\,2{\ell}-1},\underline{w}_{\,2{\ell} -1-i}^{\,i},z_{\ell})\\
       \zetas(\underline{w}_{\,1}^{\,2\ell-2-i})&\zetas(\underline{w}_{\, i+1}^{\,2\ell-1})&\zetas(x_1,\underline{w}_{\,i+1}^{2\ell-1})\nonumber\\
       0&1&\zetas(x_1)
    \end{vmatrix}\right]\\
      &=\sum_{W=S}{\bigg[}\zetas(\underline{w}_{\,1}^{\,2\ell-1},z_{\ell})\zetas(x_1)-\zetas(\underline{w}_{\,1}^{\,\ell-1})\zetas(x_1,\underline{w}_{\,\ell}^{\, 2\ell -1})
      \\
      &\left.+\zetas(x_1)\sum_{i=\ell-1}^{2\ell-2}\begin{vmatrix}
     \zetas(\underline{w}_{\,1}^{\,i},z_\ell)&\zetas(\underline{w}_{\,i+1}^{\,2\ell-1},\underline{w}_{\,2\ell-1-i}^{\,i},z_{\ell})\\
       \zetas(\underline{w}_{\,1}^{\, 2\ell-2-i})&\zetas(\underline{w}_{\ i+1}^{\, 2\ell-1})
    \end{vmatrix}\right]\\
    &=\sum_{W=S}\zetas(\underline{w}_{\, 1}^{\, \ell})\zetas(x_1)\zetas(\underline{w}_{\, \ell+1}^{\, 2\ell-1})-\zetas(x_1, \underline{w}_{\, 1}^{\, \ell})\zetas(\underline{w}_{\, \ell+1}^{\, 2\ell-1}),
 \end{align*} 
 where the notation $\underline{w}_{\, a}^{\, a+b}$ means the sequence $w_a, w_{a+1}, w_{a+2}\cdots, w_{a+b}$ as same before.
 Thus, we obtain the desired assertion.
\end{proof}
\begin{exam}[$(2,1)\cdot(1)$]For {$\Re(x_1),\Re(y_1),\Re(y_2),\Re(z_1)>1$, it holds that}
\[
\ytableausetup{boxsize=1.5em}
\sum_{sym}\begin{ytableau}
  y_1&y_2\\
  x_1
\end{ytableau}\cdot\begin{ytableau}
 *(gray) z_1
\end{ytableau}=\sum_{sym}\begin{ytableau}
 *(gray) z_1&y_2\\
  y_1\\
  x_1
\end{ytableau}+\begin{ytableau}
  y_1&*(gray)z_1\\
  x_1&y_2
\end{ytableau}+\begin{ytableau}
  y_1&y_2&*(gray)z_1\\
  x_1
\end{ytableau},
\]
where $\displaystyle{\sum_{sym}}$ means the summation over the permutation of $S=\{y_{1},y_{2},z_1\}$ as indeterminates.
\end{exam}
Considering the transpose, we obtain the other type of the Pieri formula.
\begin{theorem}[$(\lambda=({\ell}+1,\{1\}^{k-1}))\cdot (\pmb{e}_{m}=(\{1\}^{m}))$]
For a positive integer $k$ and non-negative integers $\ell$ and $m$,
% { if $\Re(x_1)>1,\ldots,\Re(x_k)>1, \Re(y_\ell)>1, \Re(z_1)>1,\ldots,\Re (z_{k-1})>1$ and $\Re(z_m)>1$,}
%{if $\Re(x_i)>1\ (1\le i\le k),\Re(y_j)\ge1\ (1\le j\le \ell-1),\Re(y_\ell)>1, \Re(z_p)>1\ (1\le p\le k-1),\Re(z_q)\ge1\ (k\le q\le m-1)$ and $\Re(z_m)>1$,}
{we assume $\Re(x_i),\Re(y_\ell), \Re(z_j)>1\ (1\le i\le k,1\le j\le k-1, j=m)$ and the real parts of other variables are greater than or equal to $1$. Then it holds that}
\[
\ytableausetup{boxsize=1.5em}
\sum_{sym}\begin{ytableau}
  x_1&y_1&\cdots&y_{{\ell}}\\
  x_2\\
  \vdots\\
  x_k
\end{ytableau}\cdot\begin{ytableau}
  z_1\\
  z_2\\
  \vdots\\
  z_m
\end{ytableau}=\sum_{sym}\sum_{\pmb u_\mu\in \pmb U_E}\zeta_{\mu}(\pmb u_\mu),
\]
where $\displaystyle{\sum_{sym}}$ means the summation over the permutation of $S=\{x_{1},\ldots,x_{{k}},z_1,\ldots,z_{{k}-1}\}$ as indeterminates.
\end{theorem}
As one of corollaries, we can also obtain Corollary \ref{216}.
\begin{cor}
For positive integers $k$ and $\ell$, and a non-negative integer $m$
{we assume that $\Re(x_k),\Re(y_\ell),\Re(z_m),\Re(a)>1$ and the real parts of other variables are greater than or equal to $1$. Then it holds that}
%, then for $a\in\mathbb C$ with $\Re(a)>1$}
\[
\ytableausetup{boxsize=1.5em}
\zeta_{(\ell,\{1\}^{k-1})}\left(\begin{ytableau}
  a&a&\cdots&a\\
  x_2\\
  \vdots\\
  x_k
\end{ytableau}\right)\cdot\begin{ytableau}
  a&\cdots&a&z_{\ell}&\cdots&z_m
\end{ytableau}=\sum_{\pmb u_\mu\in\pmb U_H}\zeta_{\mu}(\pmb u_\mu),
\]
and
\[
\ytableausetup{boxsize=1.5em}
\zeta_{(\ell,\{1\}^{k-1})}\left(\begin{ytableau}
  a&y_2&\cdots&y_{\ell}\\
  a\\
  \vdots\\
  a
\end{ytableau}\right)\cdot\begin{ytableau}
  a\\
  \vdots\\
  a\\
  z_{k}\\
  \vdots\\
  z_{m}
\end{ytableau}=\sum_{\pmb u_\mu\in\pmb U_E}\zeta_{\mu}(\pmb u_\mu).
\]
\end{cor}
%${\pmb s}=\{s\}^{\lambda}$ ($s\in\mathbb{C}$) where $\{s\}^{\lambda}=(s_{ij})\in T(\lambda,\mathbb{C})$ is the tableau given by $s_{ij}=s$ for all $(i,j)\in D(\lambda)$.
\section{Further discussion}\label{further}
{In \cite{npy}, the authors pointed out that for $\{s\}^{\lambda}:=(s_{ij})\in T(\lambda,\mathbb C)$ with $s_{ij} = s$ and $\Re(s)>1$,
\begin{equation}
    \label{identify}
    \zeta_{\lambda}(\{s\}^{\lambda})=s_\lambda(1^{-s},2^{-s},\ldots),
\end{equation}
where $s_{\lambda}$ is the Schur function with $\lambda$. Combining this with the original Pieri formulas for Schur function, we obtain the following theorem:
\begin{theorem}
Let $\lambda$ be a partition. Then, for $s\in\mathbb C$ with $\Re (s)>1$, the followings hold
\[
\ytableausetup{boxsize=1.5em}
\zeta_{\lambda}(\{s\}^\lambda)\cdot\zeta_{(m)}\left(\begin{ytableau}
  s&\cdots&s
\end{ytableau}\right)=\sum_{{\pmb u}_{\mu} \in {\pmb U}_H}\zeta_{\mu}({\pmb u}_{\mu})
\]
and
\[
\ytableausetup{boxsize=1.5em}
\zeta_{\lambda}(\{s\}^\lambda)\cdot\zeta_{(\{1\}^m)}\left(\begin{ytableau}
  s\\\vdots\\s
\end{ytableau}\right)=\sum_{{\pmb u}_{\mu} \in {\pmb U}_E}\zeta_{\mu}({\pmb u}_{\mu}).
\]
\end{theorem}
More generally, the Littlewood-Richardson rule for $\zeta_{\lambda}(\{s\}^\lambda)$ also holds.
We can confirm this from the original Littlewood-Richardson rule
and identity (\ref{identify}).

On the other hand,}
we need to consider the general case of Pieri formula for the Schur multiple zeta functions. 
However, as we have essentially difficulty, {we obtain only the following case:}
\begin{theorem}[$(\lambda=(m,2,\{1\}^{X-2}))\cdot(\pmb{h}_{\ell}=(\ell))$]
For any positive integers $m$ and $\ell$, and $X\ge2$, 
{we assume $\Re(s_{1i}),\Re(s_{21}),\Re(s_{22}), \Re(z_j),\Re(s_{X1})>1\ (1\le i\le m,1\le j\le m-1, j=\ell)$ and the real parts of other variables are greater than or equal to $1$. Then it holds that}
%the following formula holds:
%{If $\Re (s_{1j})>1,\ (1\le j\le m),\Re (s_{21})>1,\Re (s_{22})>1,\Re(s_{i1})\ge1\ (3\le i\le X-1),\Re(s_{X1})>1, \Re (z_p)>1\ (1\le p\le m-1), \Re (z_q)\ge1\ (m\le q\le \ell-1),$ and $\Re(z_\ell)>1$,}
\[
    \ytableausetup{boxsize=1.5em}
\sum_{sym}\begin{ytableau}
  s_{11}&s_{12}&\cdots&\cdots&\cdots&s_{1m}\\
s_{21}&s_{22}\\
\vdots\\
s_{X1}
\end{ytableau}\cdot \begin{ytableau}
  z_{1}&\cdots&z_{{\ell}}
\end{ytableau}=\sum_{sym}\sum_{\pmb u_\mu\in\pmb U_H}\zeta_{\mu}(\pmb u_\mu),
\]
where $\displaystyle{\sum_{sym}}$ means the summation over the permutation of $S=\{s_{11},\ldots,s_{1{m}},s_{21},s_{22},z_1,\ldots,z_{{m}-1}\}$ as indeterminates.
\end{theorem}
\begin{proof}
It suffices to show the case of $(m,X)=({\ell},3)$. By Theorem \ref{mnx}, the left-hand side of the assertion is 
 \[
    \sum_{W=S}\begin{vmatrix}
    \zetas(\underline{w}_{\,1}^{\, \ell})&\zetas(w_{{\ell}+1},w_{{\ell}+2}, \underline{w}_{\, 2}^{\, {\ell}})&\zetas(s_{31},w_{{\ell}+1},w_{{\ell}+2}, \underline{w}_{\,2}^{\, {\ell}})\\
    \zetas(w_1)&\zetas(w_{{\ell}+1},w_{{\ell}+2})&\zetas(s_{31},w_{{\ell}+1},w_{{\ell}+2})\\
     0&1&\zetas(s_{31})
    \end{vmatrix}\zetas(\underline{w}_{\, {\ell}+3}^{\, 2{\ell}+1},z_{\ell}),
 \]
 where $W=\{w_1,\ldots,w_{2l+1}\}$.
On the other hand, using Theorem \ref{mnx} and Theorem \ref{mnx2}, we find that the right-hand side of the assertion is 
 \begin{align*}
    &\sum_{W=S}\sum_{j=0}^{{\ell}-2}\left[\begin{vmatrix}
     \zetas(\underline{w}_{\, 1}^{\, 2{\ell}-1-j},z_{\ell})&\zetas(\underline{w}_{\, 2{\ell}-j}^{\, 2{\ell}+1}, \underline{w}_{\, j+2}^{\, 2{\ell}-1-j},z_{\ell})& \zetas(s_{31},\underline{w}_{\, 2{\ell}-j}^{\, 2{\ell}+1}, \underline{w}_{\, j+2}^{\, 2{\ell}-1-j},z_{\ell})\\
     \zetas(\underline{w}_{\, 1}^{\, j+1})&\zetas(\underline{w}_{\, 2{\ell}-j}^{\, 2{\ell}+1})&\zetas(s_{31}, \underline{w}_{\, 2{\ell}-j}^{\, 2{\ell}+1})\\
       0&1&\zetas(s_{31})
      \end{vmatrix}\right.\\
           %%%%%%%
      %%%%%%%
      %%%%%%%
      %%%%%%%
    &+\left.\begin{vmatrix}
     \zetas(\underline{w}_{\, 1}^{\, 2{\ell}-2-j}, z_{\ell})&\zetas( \underbrace{\underline{w}_{\, 2{\ell}-1-j}^{\, 2{\ell}}, \underline{w}_{\, j+2}^{\, 2{\ell}-2-j},z_{\ell}}_{\pmb{w}_1})&\zetas(w_{2{\ell}+1}, \pmb{w}_1) &\zetas(s_{31},w_{2{\ell}+1}, \pmb{w}_1)\\
     \zetas(\underline{w}_{\, 1}^{\, j+1})&\zetas(\underline{w}_{\, 2{\ell}-1-j}^{\, 2{\ell}})&\zetas(w_{2{\ell}+1}, \underline{w}_{\, 2{\ell}-1-j}^{\, 2{\ell}})&\zetas(s_{31},w_{2{\ell}+1}, \underline{w}_{\, 2{\ell}-1-j}^{\, 2{\ell}})\\
       0&1&\zetas(w_{2{\ell}+1})&\zetas(s_{31},w_{2{\ell}+1})\\
      0&0&1&\zetas(s_{31})
      \end{vmatrix}\right.\\
      &+
      \left.\begin{vmatrix}
     \zetas(\underline{w}_{\, 1}^{2{\ell}-2-j}, z_{\ell})&\zetas(\underline{w}_{\, 2{\ell}-1-j}^{2{\ell}}, \underline{w}_{\, j+2}^{2{\ell}-2-j}, z_{\ell})&\zetas(s_{31},w_{2{\ell}+1}, \underline{w}_{\, 2{\ell}-j}^{2{\ell}}, \underline{w}_{\, j+2}^{2{\ell}-2-j}, z_{\ell})\\
     \zetas(\underline{w}_{\,1}^{\, j+1})& \zetas(\underline{w}_{\, 2{\ell}-1-j}^{\, 2{\ell}})
     &\zetas(s_{31},w_{2{\ell}+1}, \underline{w}_{\, 2{\ell}-j}^{\, 2{\ell}})
     \\
      1&\zetas(w_{2{\ell}-1-j})&\zetas(s_{31},w_{2{\ell}+1})
      \end{vmatrix}\right.\\
      &-\left.\begin{vmatrix}
     \zetas(\underline{w}_{\, 2{\ell}-1-j}^{\, 2{\ell}}, \underline{w}_{\, j+2}^{\, 2{\ell}-2-j}, z_{\ell})&\zetas(s_{31},w_{2{\ell}+1}, \underline{w}_{\, 2{\ell}-j}^{2{\ell}}, \underline{w}_{\, j+2}^{2{\ell}-2-j}, z_{\ell})\\
     \zetas(\underline{w}_{\, 2{\ell}-1-j}^{\, 2{\ell}})&\zetas(s_{31},w_{2{\ell}+1}, \underline{w}_{\, 2{\ell}-j}^{\, 2{\ell}})
      \end{vmatrix}\right.\\
       &+\left.\begin{vmatrix}
     \zetas(w_{2{\ell}-1-j}, \underline{w}_{\, 1}^{\, 2{\ell}-2-j},z_{\ell})&\zetas(s_{31},w_{2{\ell}+1}, \underline{w}_{\, 2{\ell}-j}^{2{\ell}}, \underline{w}_{\, j+2}^{2{\ell}-2-j}, z_{\ell})\\
     \zetas(w_{2{\ell}-1-j}, \underline{w}_{\, 1}^{\, j+1})&\zetas(s_{31},w_{2{\ell}+1}, \underline{w}_{\, 2{\ell}-j}^{\, 2{\ell}})
      \end{vmatrix}\right.\\
    &+\begin{vmatrix}
     \zetas(\underline{w}_{\, 1}^{\, 2{\ell}-3-j},z_{\ell})&\zetas(\underbrace{\underline{w}_{\, 2{\ell}-2-j}^{\, 2{\ell}-1}, \underline{w}_{\, j+2}^{\,2{\ell}-3-j},z_{\ell}}_{\pmb{w}_2})&\zetas(w_{2{\ell}},w_{2{\ell}+1}, \pmb{w}_2)&\zetas(s_{31},w_{2{\ell}},w_{2{\ell}+1}, \pmb{w}_2)\\
     \zetas(\underline{w}_{\, 1}^{\, j+1})&\zetas(w_{2{\ell}-2-j}, \underline{w}_{\, 2{\ell}-1-j}^{\, 2{\ell}-1})&\zetas(\underbrace{w_{2{\ell}},w_{2{\ell}+1}, \underline{w}_{\, 2{\ell}-1-j}^{\, 2{\ell}-1}}_{\pmb{w}_{3}})&\zetas(s_{31},\pmb{w}_{3})\\
       1&\zetas(w_{2{\ell}-2-j})&\zetas(w_{2{\ell}},w_{2{\ell}+1})&\zetas(s_{31},w_{2{\ell}},w_{2{\ell}+1})\\
      0&0&1&\zetas(s_{31})
      \end{vmatrix}\\
      &-\begin{vmatrix}
    \zetas(w_{2{\ell}-2-j}, \underline{w}_{\, 2{\ell}-1-j}^{\, 2{\ell}-1}, \underline{w}_{\, j+2}^{\, 2{\ell}-3-j},z_{\ell})&\zetas(w_{2{\ell}},w_{2{\ell}+1}, \underline{w}_{\, 2{\ell}-1-j}^{\, 2{\ell}-1}, \underline{w}_{\, j+2}^{\, 2{\ell}-3-j},z_{\ell})&\zetas(s_{31},\pmb{w}_2)\\
     \zetas(\underline{w}_{\, 2{\ell}-2-j}^{\, 2{\ell}-1})&\zetas(w_{2{\ell}},w_{2{\ell}+1}, \underline{w}_{\, 2{\ell}-1-j}^{\, 2{\ell}-1})&\zetas(s_{31},\pmb{w}_{3})\\
      0&1&\zetas(s_{31})
      \end{vmatrix}\\
      &+\left.\begin{vmatrix}
    \zetas(w_{2{\ell}-2-j}, \underline{w}_{\,1}^{\, 2{\ell}-3-j},z_{\ell})&\zetas(w_{2{\ell}},w_{2{\ell}+1}, \underline{w}_{\, 2{\ell}-1-j}^{\, 2{\ell}-1}, \underline{w}_{\, j+2}^{\, 2{\ell}-3-j},z_{\ell})&\zetas(s_{31},\pmb{w}_2)\\
     \zetas(w_{2{\ell}-2-j}, \underline{w}_{\, 1}^{\, j+1})&\zetas(w_{2{\ell}},w_{2{\ell}+1}, \underline{w}_{\, 2{\ell}-1-j}^{\, 2{\ell}-1})&\zetas(s_{31},\pmb{w}_{3})\\
      0&1&\zetas(s_{31})
      \end{vmatrix}\right]\\
      &=\sum_{W=S}\sum_{j=0}^{{\ell}-2}{\bigg[}\zetas(\underline{w}_{\, 1}^{\, 2{\ell}-2-j},z_{\ell})\zetas(\underline{w}_{\, 2{\ell}-1-j}^{\, 2{\ell}})\zetas(w_{2{\ell}+1})\zetas(s_{31})-\zetas(\underline{w}_{\, 1}^{\, 2{\ell}-3-j},z_{\ell})\zetas(\underline{w}_{\, 2{\ell}-2-j}^{\, 2{\ell}})\zetas(w_{2{\ell}+1})\zetas(s_{31})\\
      &+\zetas(\underline{w}_{\, 1}^{\, 2{\ell}-3-j},z_{\ell})\zetas(\underline{w}_{\, 2{\ell}-2-j}^{\, 2{\ell}-1})\zetas(w_{2{\ell}},w_{2{\ell}+1})\zetas(s_{31})-\zetas(\underline{w}_{\, 1}^{\, 2{\ell}-2-j},z_{\ell})\zetas(\underline{w}_{\, 2{\ell}-1-j}^{\, 2{\ell}-1})\zetas(w_{2{\ell}},w_{2{\ell}+1})\zetas(s_{31})\\
      &+\zetas(\underline{w}_{\, 1}^{\, 2{\ell}-3-j},z_{\ell})\zetas(w_{2{\ell}+1})\zetas(s_{31}, \underline{w}_{\, 2{\ell}-2-j}^{\, 2{\ell}})-\zetas(\underline{w}_{\, 1}^{\, 2{\ell}-2-j},z_{\ell})\zetas(w_{2{\ell}+1})\zetas(s_{31}, \underline{w}_{\, 2{\ell}-1-j}^{\, 2l})\\
      &+\zetas(\underline{w}_{\, 1}^{\, 2{\ell}-2-j},z_{\ell})\zetas(\underline{w}_{\, 2{\ell}-1-j}^{\, 2{\ell}-1})\zetas(s_{31},w_{2{\ell}},w_{2{\ell}+1})-\zetas(\underline{w}_{\, 1}^{\, 2{\ell}-3-j},z_{\ell})\zetas(\underline{w}_{\, 2{\ell}-2-j}^{\, 2{\ell}-1})\zetas(s_{31},w_{2{\ell}},w_{2{\ell}+1})]\\
    &=\sum_{W=S}\begin{vmatrix}
    \zetas(\underline{w}_{\, 1}^{\, \ell})&\zetas(w_{{\ell}+1},w_{{\ell}+2}, \underline{w}_{\, 2}^{\, {\ell}})&\zetas(s_{31},w_{{\ell}+1},w_{{\ell}+2}, \underline{w}_{\, 2}^{\, {\ell}})\\
    \zetas(w_1)&\zetas(w_{{\ell}+1},w_{{\ell}+2})&\zetas(s_{31},w_{{\ell}+1},w_{{\ell}+2})\\
     0&1&\zetas(s_{31})
    \end{vmatrix}\zetas(\underline{w}_{\, {\ell}+3}^{\, 2{\ell}+1},z_{\ell}).
 \end{align*} 
Thus, we obtain the desired assertion.
\end{proof}

\begin{exam}[$(3,2,1)\cdot (1)$] {For $(i,j)\in D((3,2,1))$, $\Re (s_{ij})>1$ and $\Re (z_1)>1$, then}
\[
    \sum_{sym}\begin{ytableau}
  s_{11}&s_{12}&s_{13}\\
  s_{21}&s_{22}\\
s_{31}
\end{ytableau}\cdot \begin{ytableau}
 *(gray) z_{1}
\end{ytableau}=\sum_{sym}\begin{ytableau}
  *(gray)z_1&s_{12}&s_{13}\\
  s_{11}&s_{22}\\
s_{21}\\
s_{31}
\end{ytableau}+\begin{ytableau}
  s_{11}&*(gray)z_1&s_{13}\\
  s_{21}&s_{12}\\
s_{31}&s_{22}
\end{ytableau}+\begin{ytableau}
  s_{11}&s_{12}&*(gray)z_1\\
  s_{21}&s_{22}&s_{13}\\
s_{31}
\end{ytableau}+\begin{ytableau}
  s_{11}&s_{12}&s_{13}&*(gray)z_1\\
  s_{21}&s_{22}\\
s_{31}
\end{ytableau},
\]
where the sum $\displaystyle \sum_{sym}$ means the summation over the permutation of $S=\{s_{11},s_{12},s_{13},s_{21},s_{22},z_1\}$ as indeterminates.
\end{exam}

\section*{Acknowledgement}
The first author was supported by Grant-in-Aid for Scientific Research (C) (Grant Number: JP18K03223).
The second author was supported by Grant-in-Aid for JSPS Research Fellow (Grant Number: JP19J10705).

\end{document}